\documentclass[twoside,10pt]{article}
\usepackage[ansinew]{inputenc}
\usepackage[T1]{fontenc}
\usepackage{amsmath}
\usepackage{amsthm}
\usepackage{amscd, amsfonts, amssymb}
\usepackage[all]{xy}
\usepackage{mathrsfs}
\usepackage[dvips]{graphicx}
\usepackage{epic, eepic}

\usepackage{titling}
\usepackage{titlesec}
\usepackage[dotinlabels]{titletoc}
\usepackage{abstract}
\usepackage{mathptmx}       
\usepackage{helvet}         
\usepackage{courier}        
\usepackage{fix-cm}

\DeclareSymbolFont{cmletters}{OML}{cmm}{m}{it}
\DeclareSymbolFont{cmsymbols}{OMS}{cmsy}{m}{n}
\DeclareSymbolFont{cmlargesymbols}{OMX}{cmex}{m}{n}
\DeclareMathSymbol{\myjmath}{\mathord}{cmletters}{"7C}
\DeclareMathSymbol{\myamalg}{\mathbin}{cmsymbols}{"71}
\DeclareMathSymbol{\mycoprod}{\mathop}{cmlargesymbols}{"60}
\let\jmath\myjmath
\let\amalg\myamalg
\let\coprod\mycoprod

\numberwithin{equation}{section}

\theoremstyle{definition}
\newtheorem{dfn}{Definition}[section]
\newtheorem{example}{Example}[section]

\newtheorem{remark}{Remark}[section]
\theoremstyle{definition}
\newtheorem{prop}{Proposition}[section]
\newtheorem{thm}[prop]{Theorem}
\newtheorem{conj}{Conjecture}
\newtheorem{corollary}[prop]{Corollary}
\newtheorem{lem}[prop]{Lemma}
\newtheorem*{thm*}{Theorem}

\newtheorem{feeling}{Feeling}

\newcommand{\Z}{\mathbb{Z}}

\newcommand{\N}{\mathbb{N}}
\newcommand{\F}{\mathbb{F}}

\newcommand{\rrangle}{\rangle\!\!\!\rangle}

\newcommand{\llangle}{\langle\!\!\!\langle}

\newcommand{\Lc}{\mathscr{L}}
\newcommand{\Ac}{\mathscr{A}}

\newcommand{\Fc}{\mathscr{F}}
\newcommand{\Oc}{\mathscr{O}}

\newcommand{\Com}{\mathsf{Com}}
\newcommand{\Mod}{\mathsf{Mod}}
\newcommand{\Sch}{\mathsf{Sch}}

\newcommand{\ncspec}{\mathtt{Spc}}
\newcommand{\Rhat}{\widehat{\mathscr{R}}}
\newcommand{\Bhat}{\widehat{\mathscr{B}}}
\newcommand{\azu}{\mathtt{azu}}
\newcommand{\ram}{\mathtt{ram}}
\newcommand{\Simp}{\mathsf{Simp}}
\newcommand{\Frac}{\mathrm{Frac}}
\newcommand{\Spec}{\mathtt{Spec}}

\newcommand{\SDer}{\mathscr{D}er}
\newcommand{\SEnd}{\mathcal{E}{nd}}
\newcommand{\eps}{\underline{\varepsilon}}

\newcommand{\pp}{\mathfrak{p}}
\newcommand{\KW}{\mathrm{KW}}
\newcommand{\eKW}{KW}
\newcommand{\Jac}{\mathrm{Jac}}
\newcommand{\eJac}{J}
\newcommand{\W}{\mathrm{W}}
\newcommand{\eW}{W}

\DeclareMathOperator{\id}{id}
\DeclareMathOperator{\Mor}{Mor}

\DeclareMathOperator{\ob}{ob}

\DeclareMathOperator{\Ann}{Ann}
\DeclareMathOperator{\Aut}{Aut}
\DeclareMathOperator{\End}{End}
\DeclareMathOperator{\Hom}{Hom}
\DeclareMathOperator{\Ext}{Ext}
\DeclareMathOperator{\im}{im}
\DeclareMathOperator{\Der}{\mathrm{Der}}
\DeclareMathOperator{\Gal}{Gal}
\DeclareMathOperator*{\dirlim}{\underset{\longrightarrow}{\lim}}

\begin{document}
\pretitle{\begin{flushleft}\LARGE\sffamily} 
\posttitle{\par\end{flushleft}\rule[8mm]{\textwidth}{0.1mm}}
\preauthor{\begin{flushleft}\large\scshape\vspace{-8mm}}
\postauthor{\end{flushleft}\vspace{-8mm}}
\title{Arithmetic hom-Lie algebras, twisted derivations and non-commutative arithmetic schemes} 
\author{Daniel Larsson\\
\vspace{0.2cm}
\small Buskerud and Vestfold University College\\
  Pb. 235, 3603
 Kongsberg, Norway\\
 \textnormal{\texttt{daniel.larsson@hbv.no}}
 \normalsize\normalfont}
\date{}

\maketitle
\begin{abstract}
\noindent Hom-Lie algebras are non-associative algebras
generalizing Lie algebras by twisting the Jacobi identity by an
endomorphism. The main examples are algebras of twisted
derivations (i.e., linear maps with a generalized Leibniz rule). Such
generalized derivations seem to pop up in different guises in many parts of 
number theory and arithmetic geometry. In fact, any place something like $\id-\phi$, where $\phi$ is (possibly extended to) a ring morphism, appears, such as in $p$-adic Hodge theory, Iwasawa theory, e.t.c., there is a twisted derivation hiding.  Therefore, hom-Lie algebras appear to have a natural r\^ole to play in many number-theoretical disciplines. This paper is a first step in a study of these operators and associated algebras in an arithmetic-/geometric context.
\end{abstract}

\section{Introduction} 
The usefulness of automorphisms in arithmetic (by which we mean (algebraic) number theory, arithmetic geometry e.t.c) can certainly never be over-exaggerated. Indeed, it can be well argued that automorphisms constitutes the beating heart of arithmetic when it comes to supplying vital tools for the study of specific arithmetic structures, be it number fields, arithmetic schemes, zeta and $L$-functions, motives, in an endless list of topics.  Therefore, the study of the automorphisms themselves and how they relate to the underlying structure (and other structures for that matter) is extremely interesting. In fact, having a rich toolbox of structures where automorphisms of relevant objects appear in different guises is highly desirable. 

In this paper we propose
  one such toolbox, called ``hom-Lie algebras'', introduced in
  \cite{HaLaSi} in the context of quantizations of
  infinite-dimensional Lie algebras (in particular, the Witt--Lie and
  Virasoro algebras) and other ``$q$-deformations'' of Lie algebras (finite or infinite-dimensional). 

These hom-Lie algebras are non-associative, non-commutative
algebras that generalize and, so to speak, are infinitely close to
being Lie algebras. The differing property between hom-Lie algebras
and Lie algebras, is that the Jacobi identity is twisted by a linear
morphism. 

In fact, hom-Lie algebras was introduced as follows \cite{HaLaSi}:
\begin{dfn}\label{dfn:classicalhomLie}
Let $L$ be a vector space over a field $\ell$ of characteristic zero and
$\alpha$ an $\ell$-linear map on $L$. Then a \emph{hom-Lie algebra structure} on
$L$ is an $\ell$-bilinear, skew-symmetric product $\llangle
\cdot,\cdot\rrangle$ satisfying the twisted Jacobi-identity
$$\circlearrowleft_{a,b,c}\big(\llangle \alpha(a)+a,\llangle
b,c\rrangle\rrangle\big)=0.$$
\end{dfn} 
In this paper we will globalize and significantly generalize this definition and show how this
notion might be relevant to number theory.  

Since their introduction in \cite{HaLaSi}, hom-Lie algebras have
generated some interest in variou contexts
(see \cite{Jin,MakSi,YauEnv}, for instance), but in number theory they
have so far evaded attention. This is certainly understandable on one
hand since hom-Lie algebras was introduced for a different purpose, but on the
other hand a shame since they are ``morphism-like Lie algebras'' and
therefore seem extremely well suited for arithmetic. 
 
Just as Lie algebras was initially studied as algebras of derivations, hom-Lie algebras saw their day as algebras of twisted derivations. In this context a twisted
derivation on a $k$-algebra $A$ is a $k$-linear map $\partial$ satisfying a
twisted Leibniz rule:
$$\partial(ab)=\partial(a)b+\sigma(a)\partial(b)$$ for a $k$-algebra
endomorphism $\sigma$. 

Twisted derivations in the realm of arithmetic is nothing original
per se. For instance, A. Buium has introduced what he calls
``$\pi$-derivation operators'' in an attempt to develop a suitable
differential calculus in arithmetic geometry
\cite{BuiumP-adic,BuiumArDer}. In another, similar 
vein, L. di Vizio \cite{diVizio}, J. Sauloy \cite{Sauloy} and Y. Andr\'e  \cite{Andre}, among
others, has studied $q$-difference equations, 
these being equations built from twisted derivation operators. In
fact, one of the foundational reasons for introducing hom-Lie algebras
was as a way to study ($q$-) difference-type operators and their
representations in a
``Lie-algebra-like'' environment. These structures, besides being
beautiful in themselves, are, as is
hopefully amply motivated by the present paper, natural structures for
Arithmetic in its various incarnations (arithmetic geometry, Galois
representations, e.t.c). 

\subsection{Philosophy}
Let me spend a few moments commenting on the philosophy behind the above construction in the context of arithmetic.

Assume for simplicity that we are given an \emph{abelian} group scheme $G_{/R}$ over a ring $R$. Then it can, as in Lie theory, be argued that the Lie algebra to $G$ should be something like $\log G$ and this should give us derivations on the ring of functions on $G$. Now, the Taylor expansion of $\log(\sigma)$ is
$$\log(\sigma)=\sum_{i=0} (-1)^{i+1}\frac{(\id-\sigma)^i}{i},$$ and we see that the first-order term $\big(\log (\sigma)\big)_1$ is $\id-\sigma$. 

Operators on the form $a(\id-\sigma)$ are the most common type of twisted derivations, and in fact, it can be shown that on many rings \emph{all} twisted derivations are of this type (see sections \ref{sec:moduletwisted} and \ref{sec:globaltwisted}). Therefore, as twisted derivations are the most natural source of hom-Lie algebras, we notice that it is reasonable to view hom-Lie algebras as \emph{first-order Lie algebras}. 

Pushing the analogy with Lie groups and Lie algebras and their relation $\mathfrak{g}=\log(G)$, it seems reasonable to view $\big(\log(G)\big)_1$ as the ``true'' hom-Lie algebra. This is what I refer to as \emph{equivariant hom-Lie algebras} in this paper. This is because the original definition of hom-Lie algebra involved only one $\sigma$. An unfortunate result of this, and the main point where the analogy with Lie theory is flawed, is that the product in the equivariant structure is performed ``one $\sigma$ at a time''. 

We now assume that $G_{/R}$ is a finite commutative group scheme over $R$ and we put $\mathrm{homLie}_R(G)$ as the equivariant hom-Lie algebra attached to $G$ by the above construction. Clearly, we can view this as
$$\mathrm{homLie}(G)=\bigoplus_G R(\id-\sigma).$$
As a functor on affine (for simplicity) $R$-schemes we can set 
\begin{align*}
\mathrm{homLie}_R:\quad \mathsf{AffSch}_{/R}&\to \mathsf{EquiHomLie}_{/R},\\ S=\Spec(A)&\mapsto \mathrm{homLie}_R(G)(S)=\bigoplus_G A(\id-\sigma).
\end{align*}Notice that if $G=\Spec(A)$, then $A=\Mor(G,R)$ and so the equivariant hom-Lie algebra of $G$ \emph{on} $G$ should be $\bigoplus_G \Mor(G,R)(\id-\sigma)$. 

This is hom-Lie algebras from the group's perspective. However, there is another very natural perspective, namely from the perspective of the group's representations. It can be argued that the representations, not the algebras (groups) themselves, are the interesting objects in Lie theory. And in fact, it is the representation side I intend to look at in this paper. The main reason for this is that hom-Lie algebras to a very large extent \emph{arises} from group representations on commutative algebras as we will see. This was also the original view-point in the introduction of hom-Lie algebras in \cite{HaLaSi}. 

As Lie algebras measure the ``infinitesimal'' action of the Lie group on some ring, hom-Lie algebras can be said to measure the ``first-order infinitesimal'' effect of the action as the following example hopefully illustrates. 
\begin{example}Let $k$ be a complete field (for simplicity) and consider the field $k(t)$ of rational functions over $k$ in the variable $t$. Put $\sigma(t)=\epsilon t$, $\epsilon\in k$. Then $$D_\epsilon:=(1-\epsilon)^{-1}(\id-\sigma)$$ is a twisted derivation on $k(t)$ as is easly seen. The (left) $k(t)$-module $k(t)\cdot D_\epsilon$ defines a hom-Lie algebra (as we will see). Now, as $\epsilon\to 0$ one can argue successfully that $D_\epsilon\to \frac{d}{dt}$, the ordinary derivation along $t$. Therefore, choosing $\epsilon$ small enough, $\sigma$ becomes close to the identity and 
$D_\epsilon$ close to a derivation. 
\end{example}Of course, in general such a nice and clear-cut interpretation of something approaching zero, is not readily available but the intuition is still very much applicable. It is therefore natural to view the structure of hom-Lie algebras (at least the ones coming from twisted derivations) as measuring the relative effect of $\sigma\in G$, in a sense I hope to make sense of in the main text. 

However, there is one more perspective that is ever-present in twisted derivations and hom-Lie algebras, and this was in fact the \emph{true} reason (although well-hidden) for the introduction of hom-Lie algebras as algebras of twisted derivations in \cite{HaLaSi}; namely, ($\sigma$)-difference equations/operators. 

The subject of difference operators goes back centuries, but fell out of fashion during the past mid-century. Happily though, in recent time there has been a renewed interest in these kinds of operators, particularly in arithmetic. Let us briefly recall the essence. 

Classically one was primarily interested in (algebraic) function fields over $\mathbb{C}$, so we will assume this set-up below. 
\begin{example}\label{exam:qdiff}Of particular interest was (are) the following types of operators. Let $R$ be an $\mathbb{C}$-algebra and consider a (not necessarily proper) subring of $R(\!(t)\!)$.  Then
\begin{itemize}
\item[(a)] $\sigma_{(h)}(f)(t):=f(t+h)$, for any $h\in R$,
\begin{itemize}
\item[(i)] $\partial(f)(t):=(\id-\sigma_{(h)})(f(t))=f(t)-f(t+h)$, 
\item[(ii)] $\partial(f)(t):=h^{-1}\big(\id-\sigma_{(h)}\big)(f(t))=h^{-1}\big(f(t)-f(t+h)\big)$
\end{itemize}
and
\item[(b)] 
$\sigma_q(f)(t):=f(qt)$, for any $q\in R$,
\begin{itemize}
\item[(i)] $\partial(f)(t):=(\id-\sigma_{g})(f(t))=f(t)-f(qt)$,
\item[(ii)] $\partial(f)(t):=\big((1-q)t\big)^{-1}\big(\id-\sigma_q\big)(f(t))=\big((1-q)t\big)^{-1}\big(f(t)-f(qt)\big)$
\end{itemize}
\end{itemize} all define $\sigma$-derivations.
\end{example}
However, in the 70's something happend when V. Drinfel'd began studying difference operators on function fields over finite fields (where the endomorphism was a Frobenius morphism) in connection with what he called elliptic modules. From this point onward, the interest in difference operators has ever so slightly increased year-by-year. 

For instance, $q$-difference operators has been studied in arithmetic contexts since the mid 90's, for instance by the already mentioned Y. Andr\'e, L. diVizio, J. Sauloy, just to name a few. As we indicated above, the underlying reason for the paper \cite{HaLaSi} (and its antecedents \cite{LaSiQuasi, LaSi}) is the study of the algebraic structure of $q$-difference operators. A standing assumption in these papers is that the ground field is $\mathbb{C}$ or a field of characteristic zero, but this is really an unnecessary assumption. More or less every result in those papers are true in any characteristic (maybe in some cases one needs to assume that the characteristic is not $2$ or $3$). 

There is a close connection between $q$-difference operators and $q$-functions (e.g., $q$-hypergeometric functions), giving more evidence of the naturality of studying difference operators. Also, K. Kedlaya and many others (see for instance the recent book \cite{Kedlaya} by Kedlaya) study difference operators in the context of $p$-adic differential equations (Frobenius structures) and rigid cohomology. 

Therefore, it seems like a very good idea to have a ``Lie algebra-like'' structure in which to study these kind of operators.
\subsection{Plan of the paper}
The plan of the paper is as follows. As hom-Lie algebras are naturally algebras of twisted derivations, it is reasonable to begin the paper with a thorough study of these in the context of ``Global Arithmetic'', i.e., as sheaves of operators on schemes. This is done in Section 2. Section 3 introduces the main definition of this paper (besides twisted derivations), namely ``Global Equivariant hom-Lie algebras''. Then in Section 4 we give some easy results of base-change character. Assorted examples are given in Section 5. In section 6 section we introduce enveloping algebras of hom-Lie algebras which will be used in section 7 to construct non-commutative arithmetic schemes. Section 7 might be seen as a prelude to a fuller study of non-commutative scheme theory in arithmetic geometry. In this section we also construct a wealth of explicit examples and study their properties such as Auslander-regularity and representation theory. Finally we introduce the notion of zeta functions for polynomial identity algebras in order to study the arithmetic properties of the fibres of non-commutative arithmetic covers. 

This paper has a ``companion paper'', namely \cite{LarssonLfunctions}. In that paper is discussed more examples in the context of $L$-functions, Iwasawa theory and $p$-adic Hodge theory. 

\subsection*{Acknowledgments}A previous version of this paper was written when the I was employed by \emph{Forskar\-skolan i Matematik och Ber\"akningsvetenskap} (FMB) several years ago. I would like to thank FMB for funding and a great experience. I would in particular want to thank Johan Tysk for encouragement. I also want to thank Trond St\o len Gustavsen for very useful discussions. In addition, I would like to thank an anonymous referee for rejecting and giving very useful criticism and suggestions on a previous (and let's be honest, crappy) version of this paper.

\tableofcontents
\subsection*{Notations.} The following notations will be adhered to
throughout. 
\begin{itemize}
 \item[-] $k$ will denote a commutative, associative integral domain with unity. 
 \item[-] $\Com(k)$, $\Com(B)$ e.t.c, the category of, commutative,
 associative $k$-algberas ($B$-algebras, etc) with unity. Morphisms of
 $k$-algebras ($B$-algebras, e.t.c) are always unital, i.e.,
 $\phi(1)=1$. 
 \item[-] $A^\times$ is the set of units in $A$ (i.e., the set of
 invertible elements). 
 \item[-] $\Mod(A)$, the category of $A$-modules. 
 \item[-] $\End(A):=\End(A)$, the $k$-module of algebra morphisms on $A$. 
 \item[-] $\circlearrowleft_{a,b,c}(\,\cdot\,)$ will mean
cyclic addition of the expression in bracket.
\item[-] $\Sch$, denotes the category of schemes; $\Sch/S$ denotes the
  category of schemes over $S$ (i.e., the category of $S$-schemes). 
\item[-] We always assume that all schemes are Noetherian.
\item[-] When writing actions of group elements we will use the notations
  $\sigma(a)$ and $a^\sigma$, meaning the same thing: the action of $\sigma$ on $a$. 
\item[-] Sometimes we will use the notation $\underline{A}:=\Spec(A)$.
\end{itemize}
The condition that $k$ must be a domain can certainly be relaxed at several places in the presentation. But simplicity we keep it as a standing assumption. 

\section{Twisted derivations}
\subsection{Generalities}
Let $A\in\ob(\Com(k))$ and let $\sigma: A\to A$ be a $k$-linear map on
$A$. Then a
(classical) \emph{twisted derivation} on $A$ is a $k$-linear map
$\partial: A\to 
A$ satisfying $$\partial(ab)=\partial(a)b+\sigma(a)\partial(b).$$
We can generalize this as follows. Let $A$ and $\sigma$ be as above, and
$M\in\ob(\Mod(A))$. The action of $a\in A$ on $m\in M$ will be denoted
$a.m$. Then, a \emph{twisted derivation on $M$} is $k$-linear map
$\partial: M\to M$ such that 
\begin{equation}\label{eq:Leibniz}
\partial(a.m)=\partial_A(a).m+\sigma(a).\partial(m),
\end{equation} where, by
necessity, $\partial_A:A\to A$ is a twisted derivation on $A$ (in the
first sense). We call $\partial_A$ the \emph{restriction of $\partial$
  to $A$}. Finally, a \emph{twisted module derivation} is a
$k$-linear map $\partial : A\to M$ such that 
$$\partial(ab)=b.\partial(a)+\sigma(a).\partial(b),$$ for
$\sigma\in\End(A)$.  Normally we will not differentiate between left and right modules structures, but there are times when such a distinction would be necessary. 

We will sometimes refer to the above as \emph{$\sigma$-twisted
  (module) derivations} if we want to emphasize which $\sigma$ we 
  refer to. 

Let $\sigma\in\End(A)$ and denote by $A^{(\sigma)}:=A\otimes_{A,\sigma} A$, the extension of scalars along $\sigma$. This means that we consider $A$ as a left module over itself via $\sigma$, i.e., $a.b:=\sigma(a)b$. The right module structure is left unchanged. If $M$ is an $A$-module, we put 
$$M^{(\sigma)}:=A^{(\sigma)}\otimes_A M=A\otimes_{A,\sigma} M,$$i.e., $M$ is endowed with left module structure $a.m:=\sigma(a)m$, and once more, the right structure is unaffected.

We note that a $\sigma$-derivation $d_\sigma$ on $A$ is actually a derivation $d: A\to A^{(\sigma)}$ and conversely. Indeed,
$$d(ab)=d(a)b+a.d(b)=d(a)b+\sigma(a)d(b).$$In the same manner, a $\sigma$-derivation $d_\sigma: A\to M$ is a derivation $d: A\to M^{(\sigma)}$, and conversely. 

Therefore, there is a one-to-one correspondence between $\sigma$-derivations $d_\sigma:A\to M$ and derivations $d:A\to M^{(\sigma)}$. 

\subsubsection{Important note on names}I have been unable to be consistent with calling the operators by \emph{one} name. Therefore, '$\sigma$-twisted derivation', '$\sigma$-derivation' $\sigma$-difference operator' and '$\sigma$-differential operator' will all mean the same thing, unless the contrary is obvious. I hope that this will not cause the reader to much headache.

\subsubsection{Examples}
\begin{example}The ``universal'' (this designation will be amply demonstrated in what follows) example of a $\sigma$-derivation is the following. Let $A\in\ob(\Com(k))$ and
  $M\in\ob(\Mod(A))$. Suppose 
  $\sigma: M\to M$ is $\varsigma$-semilinear, i.e.,
  $\sigma(a.m)=\varsigma(a).\sigma(m)$, for $a\in A$ and $m\in M$, where
  $\varsigma\in\End(A)$. Then,
  for all $b\in A$, 
  $\partial:=b(\id-\sigma): M\to M$, is a $\varsigma$-twisted derivation on
  $M$. This follows from 
\begin{multline*}\partial(a.m)=b(\id-\sigma)(a.m)=
  b(a.m-\sigma(a.m))=b(a.m-\varsigma(a).\sigma(m))  
  \\ 
= b(a.m-\varsigma(a).m+\varsigma(a).m-\varsigma(a).\sigma(m))=
b((a-\varsigma(a)).m+\varsigma(a).(m-\sigma(m))\\
=(b(a-\varsigma(a))).m+\varsigma(a).b(m-\sigma(m))=
\partial_{A,\varsigma}(a).m+\varsigma(a).\partial(m).
\end{multline*}Notice that if $M=A$, we automatically get
  $\varsigma=\sigma$. On the other hand, given a $\sigma$-twisted derivation
  $\partial=(\id-\sigma): M\to M$, with $\sigma$ $\varsigma$-semilinear, there
  is a $\varsigma$-linear map defined by 
  $\sigma=\id-\partial$. More generally, for $a\in A^\times$,
  $\partial_a=a(\id-\sigma)$, defines a $\varsigma$-semilinear map
  $\sigma=\id-a^{-1}\partial_a$. Hence, there is a duality (in
  some sense) between twisted derivations and semilinear
  maps. Notice that this is especially true for fields. Indeed, in
  that case every twisted derivation is on the form $\partial_a$ by
  Theorem \ref{thm:classtwistder} below.
\end{example}

As was mentioned in the introduction, twisted derivations in the context of number theory is not an original
idea. A. Buium has studied arithmetic derivation-type operators since
the mid-90's in connection 
with $p$-adic abelian varieties and modular forms
\cite{BuiumP-adic,BuiumArDer,BuiumPi,BuiumMod}, where the 
present definition appears as ``$\pi$-difference operators'', being operators $\delta_\pi$ satisfying
$$\delta_\pi(x+y)=\delta_\pi(x)+\delta_\pi(y)\quad\text{and}\quad \delta_\pi(xy)=\delta_\pi(x)y+x\delta_\pi(y)-\pi\delta_\pi(x)\delta_\pi(y).$$(However, Buium's conditions involves $\delta_\pi(xy)=\delta_\pi(x)y+x\delta_\pi(y)+\pi\delta_\pi(x)\delta_\pi(y)$ instead of our $\delta_\pi(xy)=\delta_\pi(x)y+x\delta_\pi(y)-\pi\delta_\pi(x)\delta_\pi(y)$ but this is not conceptually different.) The operators in our presentation corresponding to Buium's $\delta_\pi$ are indeed the operators $\partial_\pi=\pi^{-1}(\id-\sigma)$ where $\sigma$ is the morphism $\sigma(x)= x-\pi\delta_\pi(x)$. The proof is trivial. 

The $\pi$-versions appear as a measure of defects of lifting
Frobenius from residue 
fields of discrete valuation rings, so these are fundamentally
different from our operators. Buium (cf. \cite{BuiumArDer}) also has a
variant of Theorem 4 in \cite{HaLaSi} (a 
version of which appears as Theorem \ref{thm:classtwistder} below) the
for local integral domains. 

Another instance where twisted derivations appear in number theory is
for instance as $q$-derivations (i.e., the operators from Example \ref{exam:qdiff}(b) above) and their differential calculus (e.g.,
$q$-differential equations and their dynamics, see for instance
\cite{Andre, diVizio, Sauloy}). 

\subsection{Modules of twisted derivations}\label{sec:moduletwisted}
\begin{prop}Let $M$ be an $A$-module. Then the $k$-modules of $\sigma$-twisted derivations,
\begin{align*}
\mathrm{Der}_{\sigma}(M)&:=\{\partial\in\End_k(M)\mid
  \partial(a.m)=\partial_A(a).m+\sigma(a).\partial(m)\},\quad\text{and}\\
  \mathrm{Der}_\sigma(A,M)&:=\{\partial\in\Hom_k(A,M)\mid
  \partial(ab)=\partial(a).b+\sigma(a).\partial(b)\}
\end{align*} are left 
  $A$-modules. Furthermore, if the characteristic is not 2, $\partial(1)=0$.
\end{prop}
\begin{proof}The $A$-module structure is defined, in both cases, by
  $(a.\partial)(m):=a.\partial(m)$ (for $m$ either in $M$ or in
  $A$). Since $A$ is commutative, we have 
$$(b.\partial)(a.m)=b.\partial_A(a).m+b.\sigma(a).\partial(m)=
  b\partial_A(a).m+\sigma(a).(b.\partial)(m).$$ That $\partial(1)=0$
  follows easily, noting that $\sigma(1)=1$, by the usual
  calculation. 
\end{proof} 
Note that unlike the case of ordinary derivations, $\mathrm{Der}_\sigma(M)$
or $\mathrm{Der}_\sigma(A,M)$ are not Lie algebras. 

Let, as before, $A\in\ob(\Com(k))$ and let $\sigma\in\End(A)$. Denote
by $\Delta_\sigma$ a $\sigma$-twisted derivation on $M$ whose
restriction to $A$ is $\partial$, i.e.,
$\Delta_\sigma\in\mathrm{Der}_\sigma(M)$ and
$\partial\in\mathrm{Der}_\sigma(A)$. Assume that 
$\sigma(\Ann(\Delta_\sigma))\subseteq \Ann(\Delta_\sigma)$, where
$$\Ann(\Delta_\sigma):=\{a\in A\mid a\Delta_\sigma(m)=0,\quad \text{for
  all}\quad m\in M\},$$ and that
 \begin{equation}\label{eq:qsigmapartial}
 \partial\circ\sigma=q\cdot\sigma\circ\partial,\qquad\text{for some $q\in 
  A$.}
  \end{equation} Form the left $A$-module 
$$A\cdot\Delta_\sigma:=\{a\cdot\Delta_\sigma\mid a\in A\}.$$ Define
\begin{equation}\label{eq:prodtwist}
\llangle a\cdot\Delta_\sigma,b\cdot\Delta_\sigma\rrangle:=
\sigma(a)\cdot\Delta_\sigma(b\cdot\Delta_\sigma)-
\sigma(b)\cdot\Delta_\sigma(a\cdot\Delta_\sigma). 
\end{equation}This should be interpreted as 
$$
\llangle a\cdot\Delta_\sigma,b\cdot\Delta_\sigma\rrangle(m):=
\sigma(a)\cdot\Delta_\sigma(b\cdot\Delta_\sigma(m))-
\sigma(b)\cdot\Delta_\sigma(a\cdot\Delta_\sigma(m)),$$for $m\in M$. 
We now have the following fundamental theorem.
\begin{thm}\label{thm:twistprod}Under the above assumptions,
  equation (\ref{eq:prodtwist}) gives a 
  well-defined $k$-linear product on $A\cdot\Delta_\sigma$ such that 
\begin{itemize}
\item[(i)] $\llangle a\cdot\Delta_\sigma, b\cdot\Delta_\sigma\rrangle=
  (\sigma(a)\partial(b)-\sigma(b)\partial(a))
  \cdot\Delta_\sigma$;  
\item[(ii)] $\llangle a\cdot \Delta_\sigma, a\cdot\Delta_\sigma\rrangle=0$;
\item[(iii)] $\circlearrowleft_{a,b,c}\big(\llangle \sigma(a)\cdot
  \Delta_\sigma,\llangle
  b\cdot\Delta_\sigma,c\cdot\Delta_\sigma\rrangle\rrangle+ q\cdot\llangle
  a\cdot\Delta_\sigma,\llangle
  b\cdot\Delta_\sigma,c\cdot\Delta_\sigma\rrangle\rrangle\big)=0$,
\end{itemize} where, in (iii), $q$ is the same as in (\ref{eq:qsigmapartial}).
\end{thm}
\begin{proof}See \cite{Larsson0}. 
\end{proof}
\begin{corollary}In the case $\Delta_\sigma\in\mathrm{Der}_\sigma(A,M)$, defining the algebra structure directly by property (i) in the theorem gives (ii) and (iii) on $A\cdot\Delta_\sigma$. 
\end{corollary}
\begin{proof}See \cite{Larsson0}.
\end{proof}
We can extend $\sigma$ to an algebra morphism on $A\cdot\Delta_\sigma$ by defining $\sigma(a\cdot\Delta_\sigma):=\sigma(a)\cdot\Delta_\sigma$. 
\begin{remark}\label{rem:prodtwist}
Notice that for an ideal $I\subseteq A$ and
$\Delta_\sigma\in\mathrm{Der}_\sigma(A,I)$, the module $A\cdot\Delta_\sigma$
and the product from the theorem, makes perfect sense. In particular, if $I$ is $\sigma$-stable,
$\Delta_\sigma(I)\subseteq I$ so $\Delta_\sigma$ induces a twisted
derivation $\bar\Delta_\sigma$ on $A/I$ and we can form
$(A/I)\cdot\bar\Delta_\sigma$ with induced product.  
\end{remark}
\begin{lem}\label{thm:classtwistder} If there is
  an $x\in A$, where
  $A\in\ob(\Com(k))$, such that $$x-\varsigma(x)\in A^\times,\quad \id\neq\varsigma\in\End(A),$$ then any
  $\sigma$-twisted derivation $\Delta_\sigma$ on $M$, with 
  $M\in\ob(\Mod(A))$ and 
  $$\sigma\in\End(M),\quad\sigma(a.m)=\varsigma(a).\sigma(m),$$ is on the form 
$$\Delta_\sigma=(x-\varsigma(x))^{-1}\partial_A(x)(\id-\sigma),$$where
  $\partial_A$ is the restriction of $\Delta_\sigma$ to $A$. 
  If $M$ is torsion-free over $A$, then $A\cdot
  \Delta_\sigma=\mathrm{Der}_\sigma(M)$ is free of rank one. 
\end{lem}
\begin{proof}Let $m\in M$ be arbitrary. Then the first statement follows from 
$$0=\Delta_\sigma(m.x-x.m)=\Delta_\sigma(m)(x-\varsigma(x))+
  (\sigma(m)-m)\partial_A(x).$$  By assumption, $x-\varsigma(x)$ is
  invertible, so 
$$\Delta_\sigma(m)=(x-\varsigma(x))^{-1}\partial_A(x)(\id-\sigma)(m),
  \quad \text{for all}\quad m\in M.$$ 
Clearly, when $M$ is torsion-free over $A$, $a\Delta_\sigma(m)=0\Rightarrow a=0$,
so $\mathrm{Der}_\sigma(M)$ is free of rank one. 
\end{proof}
Hence, ``up to a localization'' (at $x-\varsigma(x)$), every
$\sigma$-twisted derivation on $M\in\ob(\Mod(A))$ is on the form given
in the lemma. This means that if there is an $x\in A$ such that $x-\varsigma(x)$ is invertible, then giving a twisted derivation $\Delta_\sigma$ on $M$ amounts to deciding what the restriction of $\Delta_\sigma$ to $A$ is on $x$. 

As an immediate consequence of the lemma we have: 
\begin{prop}\label{prop:SigmaDerLocal}Let $A$ be a $k$-algebra and $\id\neq\varsigma\in\End(A)$, $\sigma\in\End_k(M)$ such that $\sigma(a.m)=\varsigma(a).\sigma(m)$. Suppose that for each $\mathfrak{p}\in\Spec(A)$ there is an $x\in A$ such that $x-\varsigma(x)\notin \mathfrak{p}$. Then $\mathrm{Der}_\sigma(M)$ is locally free of rank one over $A$.
\end{prop}
\begin{proof}
For any $\mathfrak{p}\in\Spec(A)$, take $x\in A$ such that $x-\varsigma(x)\notin\mathfrak{p}$. In the localization $A_\mathfrak{p}$ an element $x-\varsigma(x)$ is a unit so we can apply the lemma.
\end{proof}

In case $M=A$ is a unique factorization domain (UFD), it is possible (see \cite{HaLaSi}) to
prove a stronger version which does not assume the existence of $x\in
A$ such that $x-\varsigma(x)\in A^\times$:
\begin{thm}\label{thm:UFD} If $A$ is a UFD, and $\sigma\in\End(A)$, then
$$\Delta_{\sigma}:=\frac{\id-\sigma}{g}$$
generates $\mathrm{Der}_\sigma(A)$ as a left $A$-module, where
$g:=\gcd((\id-\sigma)(A))$. 
\end{thm}
Notice that the theorem and the proposition say slightly different things.
\begin{example}When $A=K/k$ is a field (extension) the above theorem
  implies that every $\sigma$-twisted derivation is on the form given
  in the statement. 
\end{example}
\subsection{Global twisted derivations}\label{sec:globaltwisted}
We keep the notations from above.

The definition of twisted derivations can be globalized. Let
$\Spec(A)$ be an affine scheme and let $E$ be an $A$-module. The $A$-module
$\mathrm{Der}_\sigma(E)$ can be ``sheafified'', i.e., turned into a sheaf $\widetilde{\mathrm{Der}_\sigma(\widetilde{E})}$ on $\Spec(A)$
(see \cite[II.5]{Hartshorne}). Let $X\xrightarrow{f} S$ be an $S$-scheme and take $\sigma\in\mathcal{E}nd_{\mathscr{O}_S}(\mathscr{O}_X)$. We first define a sheaf $\SDer_{\sigma,\mathscr{O}_S}(\mathscr{O}_X)$, the sheaf of $\mathscr{O}_S$-linear $\sigma$-twisted derivations on $\mathscr{O}_X$, as follows. We define $\SDer_{\sigma,\mathscr{O}_S}(\mathscr{O}_X)$ to be the sub-$\mathscr{O}_X$-module of $\SEnd_{\mathscr{O}_S}(\mathscr{O}_X)$ generated \emph{as a left $\mathscr{O}_X$-module} by $\mathscr{O}_X$ and the $\mathscr{O}_S$-linear operators $\partial$ in $\SEnd_{\mathscr{O}_S}(\mathscr{O}_X)$ satisfying, on each $U\subseteq X$, 
\begin{equation}\label{eq:sheafder}
\partial (xy)=\partial(x)y+\sigma(x)\partial(y),\quad x,y\in \mathscr{O}_X(U). 
\end{equation} Since, for $U=\Spec(B)\subseteq X$, $$\SDer_{\sigma,\mathscr{O}_S}(\mathscr{O}_X)\vert_U=\big(\mathrm{Der}_{\sigma,\mathscr{O}_{\underline{A}}}(\mathscr{O}_{\underline{B}})\big)^{\!\!\!\!\!\widetilde{\qquad}},$$ we see, by using \cite[Prop. II.5.4]{Hartshorne}), that $\SDer_{\sigma,\mathscr{O}_S}(\mathscr{O}_X)$ is quasi-coherent. (An alternative argument simply notes that $\SDer_{\sigma,\mathscr{O}_S}(\mathscr{O}_X)$ is a subsheaf of a quasi-coherent sheaf, and thus itself quasi-coherent.) 

 Now, suppose that $\mathscr{E}$ is a quasi-coherent sheaf of $\mathscr{O}_X$-modules. The same reasoning as above gives that $\SDer_{\sigma,\mathscr{O}_S}(\mathscr{E})$ is quasi-coherent. Notice, however, that in the definition of $\SDer_{\sigma,\mathscr{O}_S}(\mathscr{E})$, instead of (\ref{eq:sheafder}), we need to impose
 $$\partial(x.e)=\partial\vert_{\mathscr{O}_X}(x). e+\sigma(x).\partial(e),\quad x\in \mathscr{O}_X(U),\, e\in \mathscr{E}(U).$$
 
Let $G$ be a subgroup of $\mathcal{E}nd(\mathscr{O}_X)$. The set of all $\varsigma\in G$ such that $\varsigma$ reduces to the identity on the residue field $k(\mathfrak{p})$ is called the \emph{inertia group to $\mathfrak{p}$}, $\mathrm{Inert(\mathfrak{p})}$; in addition we let $\mathrm{Inert}_\varsigma(X)$ denote the set of points in $X$ where $\varsigma\in\mathrm{Inert}(\mathfrak{p})$. We also put 
$$\mathrm{Inert}_G(X):=\bigcup_{\varsigma\in G}\mathrm{Inert}_\varsigma(X),$$ the \emph{inertia locus} on $(X,G)$. We now have the following theorem.
\begin{thm}\label{thm:invsheaf}
Let $f:X\to S$ be an integral $S$-scheme. Suppose $G$ is a finite group acting on $X$ linearly over $S$. Assume that $\mathrm{Inert}_G(X)$ is a closed subscheme of $X$ and let $\mathscr{E}$ be an $G$-equivariant, torsion-free, $\Oc_X$-module. Then
\begin{itemize}
\item[(a)] $\SDer_\varsigma(\mathscr{E})$ is invertible on the complement $Y:=X\setminus \mathrm{Inert}_G(X)$, for all $\varsigma\in G$. Hence the image of the association 
$$G\to \mathrm{Pic}(Y), \quad \varsigma\mapsto \SDer_\varsigma(\mathscr{E})$$ together with the identity generates a subgroup of $\mathrm{Pic}(Y)$.
\item[(b)] if $\mathrm{Inert}_G(X)$ is regular then $\SDer_\varsigma(\mathscr{E})$ can be extended to an invertible module on all of $X$, for all $\varsigma\in G$. Hence in this case, the association becomes
$$G\to \mathrm{Pic}(X), \quad \varsigma\mapsto \SDer_\varsigma(\mathscr{E})$$ and so generates a subgroup of $\mathrm{Pic}(X)$ together with the identity.
\end{itemize} In addition, this association composes to $\mathrm{CDiv}(X)$ (resp. $\mathrm{CDiv}(X)$), generating a subgroup of effective Cartier divisors. 
\end{thm}
\begin{proof}Suppose $\{U_i\}_i$ is an affine cover of $X$ and let $\mathscr{F}$ be a quasi-coherent sheaf on $X$.  If $\mathscr{F}\vert_{U_i}$ is locally free for each $i$, then $\mathscr{F}$ is locally free. By the paragraph preceding the theorem, we know that $\SDer_\varsigma(\mathscr{E})$ is a quasi-coherent sheaf on $X$, so this applies in particular to $\SDer_\varsigma(\mathscr{E})$.  

Fix $\varsigma\in G$ and take arbitrary $\mathfrak{p}\in Y$. We restrict $f$ to $Y$. For simplicity we still denote this restriction by $f$. Then $f^{-1}(f(\mathfrak{p}))\subseteq U$ for some affine $G$-invariant  $U:=\Spec(A)\subseteq Y$. Notice that $G$ preserves the fibres above $S$ as it acts over $S$; this means that $Y$ is also $G$-invariant. Then there is an $x\in A$ such that $x-\varsigma(x)\notin\mathfrak{p}$. Indeed, either (1) we have $\varsigma(\mathfrak{p})\not\subseteq \mathfrak{p}$, or (2) we have $\varsigma(\mathfrak{p})\subseteq \mathfrak{p}$ (i.e., $\varsigma$ is in the \emph{decomposition group} at $\mathfrak{p}$). In case (1) we can take $x\in\mathfrak{p}$ such that $\varsigma(x)\notin \mathfrak{p}$; then $x-\varsigma(x)\notin\mathfrak{p}$. For case (2), assume that there is no $t\in A$ such that $(\id-\varsigma)(t)\notin \mathfrak{p}$, i.e., for all $t\in A$, $(\id-\varsigma)(t)\in \mathfrak{p}$. Then modulo $\mathfrak{p}$, $\varsigma$ reduces to the identity, which is a contradiction since we are on $Y$, and $Y$ has no points of non-trivial inertia. Therefore, for every $\mathfrak{p}\in Y$ we can choose an open affine $U=\Spec(A)$ such that there is an $x\in A$ with $x-\varsigma(x)\notin\mathfrak{p}$. 

We can now apply Proposition \ref{prop:SigmaDerLocal}, showing that $\SDer_\varsigma(\mathscr{E})$ is an invertible sheaf on $Y$. Hence we have an association $G\to \mathrm{Pic}(X)$ given by $\varsigma\mapsto \SDer_\varsigma(\mathscr{E}\vert_Y)$. For the last part, since $X$ is integral, \cite[Prop. II.6.15]{Hartshorne} states that $\mathrm{CDiv}(Y)\simeq \mathrm{Pic}(Y)$, and \cite[Rem. II.6.17.1]{Hartshorne} shows that $\SDer_\varsigma(\mathscr{E}\vert_Y)$ actually gives an effective Cartier divisor since it is locally generated by one element. This proves (a).

For (b), we simply remark that the local ring at a point on a regular scheme is regular and thus a UFD. We now apply Theorem \ref{thm:UFD} to finish the proof. 
\end{proof}

\begin{remark}The above association gives us, for each $n\in\mathbb{N}$, a map
$$\sigma^n\mapsto \SDer_{\sigma^n}(\mathscr{E})\in\mathrm{Pic}(X).$$ However, note that if $\sigma=\id$, then $\SDer_{\sigma}(\mathscr{E})=\mathrm{Der}(\mathscr{E})\notin \mathrm{Pic}(X)$, so the association can certainly not be a group morphism. 
\end{remark}
\begin{remark}It would obviously be very interesting to know what kind of subgroup the image of $G$ generates inside $\mathrm{Pic}(X)$. For instance, are there sufficient conditions that $\langle G\rangle = \mathrm{Pic}(X)$? 
\end{remark}

Let us briefly recall the definition of a tamely ramified $G$-covering. We use a slightly more restrictive definition than usual for simplicity.
\begin{dfn}Let $\pi: X\twoheadrightarrow S$ be a \emph{finite} cover with $S$ connected and normal and $X$ normal. We let $D\subset S$ denote a normal crossings divisor such that $\pi$ is \'etale over $S\setminus D$ and assume that $\pi^{-1}(D)$ is regular. Then $X\twoheadrightarrow S$ is a (\emph{tamely}) \emph{ramified extension} if for every $s\in D$ of codimension one (in $S$) and $x\in X$ such that $s=\pi(x)$, $\mathscr{O}_{X,x}/\mathscr{O}_{S,s}$ is a (tamely) ramified extension of discrete valuation rings. If, in addition, $$X\times_{S}(S\setminus D)\to S\setminus D$$ is a $G$-torsor, i.e., Galois covering with $G=\Gal(k(X)/k(S))$, then $\pi$ is a (\emph{tamely}) \emph{ramified $G$-covering}. 
\end{dfn}
\begin{example}Let $\pi: X\twoheadrightarrow S$ be a tamely ramified $G$-covering, ramified along a divisor $D$ and let $\mathscr{E}$ be a torsion-free sheaf on $X$. Then $D$ includes the points over which $\mathrm{Inert}_G(X)$ is non-zero. Therefore, the assumptions of Theorem \ref{thm:invsheaf} are satisfied and so 
$$\SDer_G\big(\mathscr{E}\vert_{X\setminus\mathrm{Inert}_G(X)}\big)$$ is a family of invertible sheaves on $X\setminus\mathrm{Inert}_G(X)$. On the other hand, since by assumption $\pi^{-1}(D)$ is regular, by Theorem \ref{thm:invsheaf}(b), we can extend $\SDer_G(\mathscr{E})$ to a family of invertible sheaves on the whole of $X$. 
\end{example}
\begin{example}As a particular case of the preceding example we can take $\pi:X\to S$ with $X=\underline{\mathfrak{o}_L}$ and $S=\underline{\mathfrak{o}_K}$ and such that $L/K$ is a Galois extension. Let $\mathscr{E}$ be a projective $\mathfrak{o}_L$-module (which is automatically torsion-free since $\mathfrak{o}_L$ is a Dedekind domain) and let $D$ be a divisor of $S$ including all the ramified primes in $X$. In other words, $D$ is a finite set of places in $\mathfrak{o}_K$ including the ramified ones in $\mathfrak{o}_L$. Natural choices for $\mathscr{E}$ are of course $\mathfrak{o}_L$ itself and fractional ideals $\mathcal{J}\in\mathrm{Pic}(\mathfrak{o}_L)$. Then on $X\setminus\pi^{-1}(D)$, $\SDer_{\Gal(L/K)}(\mathscr{E})$ is an invertible sheaf. Therefore, we have an association (dependent on $D$) 
$$\Gal(L/K)\to \mathrm{Pic}(\underline{\mathfrak{o}_L}\setminus\pi^{-1}(D)), \quad \varsigma\mapsto \SDer_\varsigma(\mathscr{E}\vert_{\underline{\mathfrak{o}_L}\setminus\pi^{-1}(D)}),$$ for every $\mathscr{E}$. However, since $\pi^{-1}(D)$ is regular Theorem \ref{thm:invsheaf}(b) applies again, and we can extend to a the whole $\mathfrak{o}_L$,
$$\Gal(L/K)\to \mathrm{Pic}(\mathfrak{o}_L), \quad \varsigma\mapsto \SDer_\varsigma(\mathscr{E}),$$ for every projective $\mathfrak{o}_L$-module $\mathscr{E}$. In fact, in this case we could argue by simply appealing to Theorem \ref{thm:UFD} directly since $\mathfrak{o}_L$, being a Dedekind domain, is automatically regular and hence every localization is a UFD. 
\end{example}
\subsubsection{$\mathscr{D}_\sigma$-modules}\label{sec:Dmodule}
Let $\mathscr{L}$ be an invertible sheaf on $X\to S$. Then there is an open cover $\{U_i\}$ of $X$ such that the $\Oc_X$-module $\mathscr{D}:=\mathrm{Sym}(\mathscr{L})$ is given locally as $\Oc_X(U)[T]$. Fix $\sigma\in\Aut_S(X)$. Now, we let $T$ act on $\Oc_X(U)$ by the rule $T. a:=(\id-\sigma)(a)$, i.e., we let $T$ act as a $\sigma$-derivation. Then $\mathscr{D}$ can naturally be viewed as $\sigma$-difference operators (equations) on $\mathscr{O}_X$. 

With this paper we would like to argue that ``arithmetic differential equations'', in the sense of Berthelot \cite{Berthelot} on the one hand, and Buium (see references from the introduction) on the other, are not differential equations at all, but \emph{difference} equations. See also the recent treatment by Kedlaya \cite{Kedlaya}. 

\section{Equivariant Hom-Lie algebras}

\subsection{Global equivariant hom-Lie algebras}

Fix a scheme $S$ and an $S$-scheme $f:X\to S$. Let $(X)_{\mathsf{top}}$ denote the (small) $\mathsf{top}$-site associated
  with $X$, where $\mathsf{top}$ denotes suitable family of morphisms (e.g., flat, \'etale, Zariski (open immersions)). To recall, this is the category of $\mathsf{top}$-morphisms $Y\to X$ over $S$,
  with the morphisms between $Y\to X$ and $Y'\to X$ being $S$-$\mathsf{top}$-morphisms. The covering
  families are families of $S$-$\mathsf{top}$-morphisms $(U_i\to Y\to X)_i$, where $i\in
  I$ for some 
  index set $I$. For details on this see \cite{Milne}.
  All morphisms are over $S$ so from now on we simply omit mentioning $S$ unless ambiguities can arise. 
  
  By $G\to S$ we denote a group scheme over $S$.   

Let $\Oc_X$ be the structure sheaf on $(X)_{\mathsf{top}}$ in the sense
that $\Oc_X(U):=H^0(U,\Oc_U)$ for
$U\in\ob((X)_{\mathsf{top}})$ and let $\Ac$
be a sheaf of $\Oc_X$-algebras. In addition, let $\Lc$ be an $\Ac$-module. Let $G$ act $\Oc_S$-linearly on $\Lc$. We don't specify in advance how $G$ acts on $\Ac$. 
\begin{dfn}\label{dfn:globhom}Given the above data, an \emph{equivariant hom-Lie
  algebra for $G$ on 
  $(X)_{\mathsf{top}}$ over $\Ac$} is a $(G-\Ac)$-module $\Lc$ together
  with, for each open $U\in (X)_{\mathsf{top}}$, an $\Oc_S$-bilinear product $\llangle\,\cdot,\cdot\,\rrangle_{U}$ on
  $\Lc(U)$ such that 
  \begin{itemize} 
       \item[(hL1.)] $\llangle a,a\rrangle_{U} =0$, for all $a\in \Lc(U)$;
       \item[(hL2.)] $\circlearrowleft_{a,b,c}\big (\llangle a^\sigma,\llangle
       b,c\rrangle_{U}\rrangle_{U}+q_\sigma\cdot\llangle a,\llangle
       b,c\rrangle_{U}\rrangle_{U}\big)=0$, for all $\sigma\in G\vert_{U}$ and some $q_\sigma\in \mathscr{A}(U)$. 
  \end{itemize} A morphism of equivariant hom-Lie algebras $(\Lc, G)$ and
       $(\Lc', G')$ is a pair $(f,\psi)$ of a morphism of
       $\Oc_X$-modules $f: \Lc\to\Lc'$ and $\psi: G\to G'$ such 
       that $f\circ \sigma=\psi(\sigma)\circ f$, and $f({U})\big(\llangle
       a,b\rrangle_{\Lc; U}\big) =\llangle  
       f({U})(a),f({U})(b)\rrangle_{\Lc'; U}$. 
\end{dfn}Notice that the definition implies that for a morphism $$(f,\psi): (\mathscr{L},G)\to (\mathscr{L}', G')$$ we must have $f(q_\sigma)=q_{\psi(\sigma)}$. 

If all the $q_\sigma=\id$ for all $\sigma\in G$ we can say that the equivariant hom-Lie algebra is \emph{strict}. 

We denote by $\mathsf{EquiHomLie}_{X_{/S}}$ denote the category of all
equivariant hom-Lie algebras on $(X)_{\mathsf{top}}$ with morphisms given in the
definition. We will sometimes use the notation $\mathscr{L}^\diamond$ as a short-hand for $(\mathscr{L},\mathscr{A},G,\llangle\,,\,\rrangle)$ where the 'diamond' is there to remind us that there are objects not specified explicitly. 

Hence, an equivariant hom-Lie algebra is a family of (possibly isomorphic) products
para\-metrized by $G$. A product $\llangle\,\cdot,\cdot\,\rrangle_g$,
for fixed $g\in G$, is a \emph{hom-Lie algebra on $\Lc$}. The category of hom-Lie algebras over $X_{/S}$ is denoted $\mathsf{HomLie}_{X_{/S}}$. 

Notice that, by the requirements that $G$ is a group, every equivariant hom-Lie algebra
includes a Lie algebra, possibly abelian, corresponding to 
$e\in G$ (see Example \ref{exam:Lie} below). The hom-Lie algebras
corresponding to $g\neq e$ in the equivariant hom-Lie algebra can be viewed as
deformations, in some weak sense, of the Lie algebra in the equivariant hom-Lie
algebra. It is not strictly necessary for the definition above to
make sense, to require that $G$ is a group. However, it is quite
convenient as we then, as was mentioned, can view a equivariant hom-Lie algebra
as a family of deformations of the Lie algebra corresponding to $g=e$. 

From now on we will suppress notation for the topology when writing the scheme $X$; anyone topology is as good as any other. 

\begin{remark}Since $G$ is a group, we require that $G$ acts as automorphisms on the sheaf $\Lc$. This is stricty not necessary for the definition to work. We could equally well have worked with monoid schemes, giving endomorphisms on the sheaf instead. But in all the examples we have in mind there is a group present so we will stick with this. 
\end{remark}
\subsubsection{Affinization}When $S=\Spec(k)$, $X$ affine over $S$ and $U\subseteq X$ an open affine, Definition
  \ref{dfn:globhom} 
  specializes to 
\begin{itemize}
\item[-] $\Oc_X\leadsto \mathfrak{o}\in\ob(\Com(k))$; 
\item[-] $\Ac\leadsto A\in\ob(\Com(\mathfrak{o}))$;
\item[-] $\llangle\,\cdot,\cdot\,\rrangle_g({U})\leadsto
  \llangle\,\cdot,\cdot\,\rrangle_g$ (only one product for each $g\in G$).
\end{itemize}

Hence, a hom-Lie algebra is then simply an $A$-module $L_g$
  with an $\mathfrak{o}$-bilinear product satisfying conditions (hL1.) and
  (hL2.) of Definition \ref{dfn:globhom} and a equivariant hom-Lie algebra $L$ is
  the disjoint union over all hom-Lie algebras, i.e., $L:=\coprod_{g\in
  G}L_g$. See Example 
  \ref{exam:special} for an explicit example. 

When we need to specify the difference of the above case and Definition
\ref{dfn:globhom}, we call this the \emph{affine case} and
\ref{dfn:globhom} the \emph{global case} (and so the global case
includes the special).

\begin{remark}
Thus affinization is in a sense a way of comparing the individual hom-Lie algebras in the equivariant structure.
\end{remark}
\section{Base change}
\begin{thm}\label{thm:basechange}Let $f:X\to Y$ be a morphism in
  $\Sch/S$ and let $\Ac$ an $\Oc_Y$-algebra on $Y$. Suppose that $\Lc$ is
  a hom-Lie algebra over $\Ac$ on $Y$. Then, 
$$f^\ast\Lc:=f^{-1}\Lc\otimes_{f^{-1}\Oc_Y}\Oc_X,\quad \text{the
  \emph{pull-back} of } \Lc,$$ is a hom-Lie algebra
  over $f^{\ast}\Ac$ on $X$. 
\end{thm} 
\begin{proof}This is standard. Every $\Lc(V)$ comes endowed with an
  $\Oc_Y$-bilinear product $\llangle\,\cdot,\cdot\,\rrangle(V)$ satisfying (hL1.) and (hL2.). Taking
  the direct limit preserves the algebra structure. This means that $\llangle \,\cdot,\cdot\,\rrangle(V)$ extends to a
  well-defined product on $\dirlim \Lc(V)$ and also to the sheafification. Hence $f^{-1}\Lc$ is a
  hom-Lie algebra. The extension to $f^\ast\Lc$ is defined by 
$$\llangle a\otimes x,a'\otimes x'\rrangle:=\llangle a,a'\rrangle\otimes
  xx',$$ with action of $g$ on
  $f^\ast\Lc$ extended by $(a\otimes x)^g:=a^g\otimes x$. 
\end{proof}
\subsection{Affinization of base change}
 Let $A$, $B$ and $C$ be $k$-algebras and $A$ and $B$, $C$-algebras. Suppose $L_A$ is a hom-Lie 
    algebra over $B$, then $L_B:=L_A\otimes_C B$ is a
    hom-Lie algebra over $B$ with bracket defined by
    $$\llangle l_1\otimes b_1,l_2\otimes b_2\rrangle_g:=\llangle
    l_1,l_2\rrangle_g\otimes b_1b_2,\quad g\in G$$ where we extend $g$ as
    $(l\otimes b)^g:=l^g\otimes b$.  
\subsection{Change of group}
We will now consider what happens when we change the group. For simplicity we consider only the special case. Everything globalizes without problem.

Let $L$ be an $A$-module, $A\in\Com(k)$ equipped with an equivariant hom-Lie algebra with group $G$. Suppose we are given a sequence of groups 
$$\cdots\rightarrow H_{i+1}\xrightarrow{\varsigma_{i+1}} H_i\rightarrow \cdots \xrightarrow{\varsigma_2} H_1\xrightarrow{\varsigma_1} H\xrightarrow{\varsigma} G\xrightarrow{\psi} E.$$ 
Then we have the following proposition.
\begin{prop}The equivariant hom-Lie algebra on $L$ over $G$ descends to a canonical one, $\varsigma^\ast L$, over $H$ via $\varsigma$; also $G$ vill act on the invariants $L^{\varsigma(H)}$ with induced equivariant hom-Lie algebra over $G/\im(\varsigma)$. We also have an induced map $L^G\to L^{\varsigma(H)}$, or more generally 
$$\cdots\leftarrow L^{\varsigma_{i+1}(H_{i+1})}\xleftarrow{} L^{\varsigma_i(H_i)}\leftarrow \cdots \xleftarrow{} L^{\varsigma_1(H_1)}\xleftarrow{} L^{\varsigma(H)}\xleftarrow{} L^G,$$with induced hom-Lie structures. Notice that $L^G$ is the trivial (abelian) equivariat hom-Lie algebra. In addition, if $\psi: G\twoheadrightarrow E$ is a surjection, then $L^{\ker\psi}\subseteq L$ is a equivariant hom-Lie algebra over $E$.
\end{prop}
\begin{proof}Obvious.
\end{proof}
Recall that the induced $G$-module coming from an $H$-module $M$, is defined as
$$\mathrm{Ind}^H_G(M):=\{\psi: G\to M\mid \psi(hg)=h\psi(g), \, \text{for } h\in H\,\}.$$ The $G$-module structure on $\mathrm{Ind}^H_G(M)$ is defined by $(g'.\psi)(g):=\psi(gg')$.
\begin{prop}Suppose that the $A$-module $L$ is an equivariant hom-Lie algebra over $H$. Then $\mathrm{Ind}^H_G(L)$ is an equivariant hom-Lie algebra over $G$ with product defined by
$$\llangle \psi,\psi'\rrangle_{\mathrm{Ind}^H_G(L)}(g):=\llangle \psi(g),\psi'(g)\rrangle_L$$and $A$-module structure given by $(a.\psi)(g):=a\psi(g)$.
\end{prop}
\begin{proof}Obvious.
\end{proof}
Notice that I allow arbitrary group morphisms when defining induced modules, contrary to the ordinary usage, which restricts to injective morphisms. In general, induced modules are only useful when $H$ is indeed a subgroup of $G$. 

\section{Examples}
\subsection{General examples}
\begin{example}\label{exam:Lie}Suppose 
    $G=\{\id\}$, the 
  trivial group. Then the above definition amounts to a sheaf of
  $\Oc_X$-Lie algebras. In the special case we simply get an ordinary
  $\mathfrak{o}$-Lie algebra. 
\end{example}
\begin{example}[Example of affinization]\label{exam:special} Let
  $A\in\ob(\Com(k))$, $L\in\ob(\Mod(A))$ and 
  $G$ a group acting $k$-linearly on $L$. Then an
  equivariant hom-Lie algebra on $L$ 
  over $A/k$ is
  family of $k$-bilinear products $\llangle\,\cdot,\cdot\,\rrangle_g$,
  $g\in G$, satisfying 
  $$\llangle a,a\rrangle_g=0\quad\text{and}\quad
  \circlearrowleft_{a,b,c}\big (\llangle a^g+a,\llangle
  b,c\rrangle_g\rrangle_g\big)=0, \quad\text{for all}\quad
  g\in G.$$ A morphism of equivariant hom-Lie algebras $L$ and $L'$ over
  $A/k$ is a morphism of $k$-modules such that $f(a^g)=f(a)^g$ (i.e.,
  $G$-equivariance) and
  $f\llangle a,b\rrangle^L_g =\llangle f(a),f(b)\rrangle_g^{L'}$. 

If $L$ and $L'$ comes equipped with different group actions, $G$ and
$G'$, we demand according to definition, instead of $G$-equivariance,
that $f(a^g)=f(a)^{g'}$, 
for all $g\neq e\in G$, $g'\neq e'\in G'$.
\end{example}
\begin{example}[Twisted derivations]\label{exam:homLietwist}Let $A\in\ob(\Com(k))$ and assume $M\in\ob(\Mod(A))$ torsion-free. Suppose that
  $\varsigma\in\End(A)$, 
  $\delta_\sigma\in\mathrm{Der}_\sigma(M)$ are such that $\partial=a(\id-\varsigma)$, $a\in A$, and
  $$\partial\circ\varsigma=q\cdot\varsigma\circ\partial,\quad\text{with}\quad q\in
  A.$$ Assume in addition that
  $$\sigma\Ann(\partial)\subseteq \Ann(\partial),$$which is automatic for instance when $A$ is a domain. Then
  Theorem \ref{thm:twistprod} endows $A\cdot\delta_\sigma$ with the
  structure of hom-Lie algebra. Taking a subgroup $G\subseteq \End(A)$ with a family $\delta_G\subseteq\mathrm{Der}_\sigma(M)$,
  $\delta_G:=\{\delta_\sigma\mid \sigma\in G\}$, such that
  $$\partial_{\varsigma}\circ\varsigma=
  q_\varsigma\cdot\varsigma\circ\partial_{\varsigma},\quad\text{with}  
  \quad q_\varsigma\in
  A,\quad\text{for each
  $\sigma\in G$},$$ and where $\sigma(am)=\varsigma(a)\sigma(m)$. Then Theorem \ref{thm:twistprod} gives us an equivariant hom-Lie
  algebra for $G$ on $M$. It is easy to see that if $a\in A^\times$, then $q_\varsigma:=a/\varsigma(a)$ satisfies the assumptions on $q_\sigma$. Indeed, 
  $$\sigma\circ\partial_{\sigma}(b)=\sigma\circ (a(\id-\sigma))(b)=\sigma(a)(\id-\sigma)\circ\sigma(b),$$ so multiplying by $a/\sigma(a)$ gives the desired identity. 
  Fixing $a\in A^\times$, we get an association $G\to A^\times$, $\sigma\mapsto a/\sigma(a)$. In other words, we get an element in $\mathrm{B}^1(G,A^\times)$. This gives a family $\{(q_\sigma,\partial_{\sigma})\mid \sigma\in G, \, \partial_{\sigma}=a(\id-\sigma)\}$ satisfying the required conditions of Theorem \ref{thm:twistprod}. 
  
  Notice that we in particular get that if $A$ is a domain, $\mathrm{Der}_G(\mathrm{Fr}(A))$ is an equivariant hom-Lie algebra, where $\mathrm{Fr}(A)$ is the fraction field of $A$. 
  
We can globalize this in the evident manner. Namely, let $X$ be a scheme, $\mathscr{A}$ a sheaf of $\mathscr{O}_X$-algebras and $\mathscr{E}$ a torsion-free $\mathscr{A}$-module. First, for $U\subseteq X$ an open affine, let $\partial$ be a section of $\mathcal{D}er_\sigma(\mathscr{A})(U)$ such that $\partial\circ \sigma =q_{\sigma,U}\cdot\sigma\circ \partial$, for some $q_{\sigma,U}\in \mathscr{A}(U)$, and $\sigma\Ann(\partial)\subseteq\Ann(\partial)$. Then to any $\delta\in\mathcal{D}er_\sigma(\mathscr{E})(U)$ such that $$\delta(am)=\partial(a)m+\sigma(a)\delta(m),$$ is attached a canonical global hom-Lie algebra, $\mathscr{A}\cdot\delta\subseteq\mathcal{D}er_\sigma(\mathscr{A})$, and therefore a global equivariant hom-Lie algebra, $\mathscr{A}\cdot\delta_G$. 
\end{example}

\begin{example}[First order terms in difference operators.]
First recall (specialized version of) the definition of $\mathscr{D}_\sigma$-modules from subsection \ref{sec:Dmodule}. We let $\mathfrak{o}$ be commutative as always and consider the polynomial algebra $\mathscr{D}_\sigma:=\mathfrak{o}[T]$. Taking $y\in \mathfrak{o}$, we let $T$ act on $\mathfrak{o}$ as $T.a:=y(\id-\sigma)(a)$. This is the (non-commutative) algebra of $\sigma$-difference operators. In the following we put $\partial_\sigma:=y(\id-\sigma)$ (suppressing $y$ from the notation) and view $\mathscr{D}$ as the ring $\mathfrak{o}[\partial_\sigma]$. 

For this example it is useful to introduce the following notations. 
Let $\partial_\sigma $ be a $\sigma$-derivation on $\mathfrak{o}$. Then we denote by $\pi^n_i$ the sum
of all permutations of $(n-i)$ mappings $\partial_\sigma $ and $i$ mappings
$\sigma$. As an example
$\pi^3_1=\partial_\sigma ^2\circ\sigma+\partial_\sigma \circ\sigma\circ\partial_\sigma +\sigma\circ\partial_\sigma ^2$.
Note in particular that $\pi^k_k=\sigma^k$ and $\pi^k_0=\partial_\sigma ^k$. We
also put $\pi^n_k=0$ for $n<k$ and $k<0$. Then one can easily prove that
\begin{equation}\label{eq:pi} 
\partial_\sigma ^n(ab)=\sum_{i=0}^n\pi^n_i(a)\partial_\sigma ^i(b)\qquad
  \text{(Leibniz' rule for $\sigma$-derivations)}.
\end{equation}
\begin{subequations}
A $\sigma$-difference operator is a linear combination
$$P(\partial_\sigma)=\sum_{i=0}^n p_i\partial_\sigma^i, \quad p_i\in \mathfrak{o}.$$ Now one can introduce a skew-symmetric product on the ring of $\sigma$-difference operators as
\begin{align*}
\llangle
a\cdot\partial_\sigma ^n,b\cdot\partial_\sigma ^m\rrangle:=
\sigma(a)\cdot\partial_\sigma ^n(b\cdot\partial_\sigma ^m)-\sigma(b)\cdot\partial_\sigma ^m(a\cdot\partial_\sigma ^n).
\end{align*}
It is possible to compute, with some work, that on monomials we have, with $n<m$
\begin{align}\label{eq:bracket1}
\begin{split}
\llangle
a\cdot\partial_\sigma ^n,b\cdot\partial_\sigma ^m\rrangle
=-\sum_{i=n}^{m-1}&\sigma(b)\pi^m_{i-n}(a)\cdot\partial_\sigma ^i\\
&+\sum_{i=m}^{n+m}\big
(\sigma(a)\pi^n_{i-m}(b)-\sigma(b)\pi^m_{i-n}(a)\big )\cdot\partial_\sigma ^i.
\end{split}
\end{align}If $n=m$ the above equation reduces to
\begin{align}\label{eq:brackH_qn=m}
\llangle a\cdot\partial_\sigma ^n,b\cdot\partial_\sigma ^n\rrangle
&=\sum_{i=0}^n(\sigma(a)\pi^n_i(b)-\sigma(b)\pi^n_i(a))\cdot\partial_\sigma ^{n+i}.
\end{align}
\end{subequations}Restricting to first-degree terms, this reduces to Example \ref{exam:homLietwist}, so hom-Lie algebras can be viewed as the linear part of a commutator-like product on difference operators, just as Lie algebras can be viewed as the linear part of differential operators under the classical commutator. 
\begin{remark}There is another way to express difference operators which is perhaps more prevalent in the literature (see \cite{Kedlaya}, for instance). Namely, a difference operator in that sense
is a formal expression 
\begin{equation}\label{eq:differenceoper}
P(\sigma)=\sum_{i=0}^n p_i\sigma^i,\quad p_i\in \mathfrak{o}.
\end{equation} You can go from this to the above, by simply putting $\partial=\id-\sigma$ and rearranging. Also, if $y$ is invertible (in particular if it is $1$) you can also go in the other direction. 
\end{remark} 
\end{example}

\subsection{Group representations}\label{sec:grouprep}

Let $(X_{/S},\mathscr{A})$ be an $S$-scheme together with a sheaf of coherent $\mathscr{O}_X$-algebras. Put $Y:=\mathbf{Spec}(\mathscr{A})$, the global spectrum of $\mathscr{A}$. Of particular interest to us is the case where $\mathscr{A}$ is a finite-rank $\mathscr{O}_X$-algebra. Assume that $G_{/S}$ is a group scheme acting $S$-linearly and equivariantly on $\mathscr{A}$. Then this induces a $G$-action on $Y_{/S}$. 

 Now take a global section $w\in \mathscr{A}$ and form $D_w:=w(\id-\sigma)$, for $\sigma\in G$. We leave $\sigma$ out of the notation whenever misunderstandings are unlikely. Then we form the left $\mathscr{A}$-module 
 $$\mathscr{A}\cdot D_w\subseteq \SDer_\sigma(\mathscr{A}).$$ On open affines $U\subseteq X$, $\mathscr{A}(U)$ is a $G$-representation. 
 
On this $\mathscr{A}$-module we introduce a hom-Lie algebra product as 
$$\llangle a\cdot D_w,b\cdot D_w\rrangle(U) := \big(\sigma(a)D_w(b)-\sigma(b)D_w(a)\big)\cdot D_w, \quad a,b\in\mathscr{A}(U).$$ Obviously, we consider $w$ restricted to $\mathscr{A}(U)$.

The above product defines a global hom-Lie algebra if $D_w\circ\sigma=q\sigma\circ D_w$, for some $q\in \mathscr{A}$. For instance, this condition is satisfied when $w\in\mathscr{A}^\times$, as we have seen. 

The following is a simple consequence of the definitions.
\begin{prop}Equivalent representations give isomorphic equivariant hom-Lie algebras. 
\end{prop}
\begin{proof}Suppose $\rho\sim \rho'$. This means that for every $U\subseteq X$ there is an isomorphism $f: \mathscr{A}(U)\to \mathscr{A}(U)$ of $\mathscr{O}_X$-algebras such that $f\circ \rho(\sigma)=\rho'(\sigma)\circ f$. Now, $$f\circ \rho(\sigma)(a)=\rho'(\sigma)\circ f(a),$$ whence
\begin{multline*}
f\circ D^{\rho}_w=f\big(w(\id-\rho(\sigma))\big)=f(w)\big(f-\rho'(\sigma)\circ f\big)=
f(w)(\id-\rho'(\sigma))\circ f=D_{f(w)}^{\rho'}.
\end{multline*}From this we calculate
\begin{align*}
	f\llangle aD_w^{\rho},bD_w^{\rho}\rrangle^\rho &=f\big(a^{\rho(\sigma)}D_w^\rho(b)-b^{\rho(\sigma)}D_w^\rho(a)\big)\\
	&=\big(f(a)^{\rho'(\sigma)}D_{f(w)}^{\rho'}(f(b))-f(b)^{\rho'(\sigma)}D_{f(w)}^{\rho'}(f(a))\big)D_{f(w)}^{\rho'}\circ f\\
	&=\llangle f(aD_w^\rho),f(bD_w^\rho)\rrangle^{\rho'},
\end{align*}finishing the proof.
\end{proof}

\subsection{Hom-Lie algebraization}
Let $X_{/S}$ be an $S$-scheme. We introduce the category $\Com^G(X_{/S})$ of quasi-coherent $\mathscr{O}_X$-algebras $\mathscr{A}$ with an action of a group $G\subseteq\End(\mathscr{A})$. Morphisms are given by 
pairs $(f,\psi)$, where $f:\mathscr{A}\to \mathscr{B}$ is a $\mathscr{O}_X$-algebra sheaf morphism and $\psi: G_A\to G_B$ a group morphism, such that $f\circ
g=\psi(g)\circ f$. By $\Com^-(X_{/S})$ we denote the category
obtained by considering as objects pairs $(\mathscr{A},\sigma)$ for
$\sigma_\mathscr{A}\in G_A$ and morphisms $f:\mathscr{A}\to\mathscr{B}$ with $f\circ \sigma_\mathscr{A}=\sigma_\mathscr{B}\circ
f$.
\begin{thm}\label{thm:AltProduct}Let $(\mathscr{A},G)\in\ob(\Com^G(X_{/S}))$.
  Then, for $\sigma\in G$, the product
\begin{equation}\label{eq:homCom}
\llangle a\cdot
\mathbf{e},b\cdot\mathbf{e}\rrangle_\sigma:=(\sigma(a) b-\sigma(b) a)\cdot\mathbf{e}, \quad a,b\in \mathscr{A}(U),\quad U\subseteq X,
\end{equation} defines a hom-Lie
  algebra product 
  on $\mathscr{A}\cdot\mathbf{e}$, where $\mathbf{e}$ is some global basis of $\mathscr{A}$ as $\mathscr{A}$-module. This defines a functor $\underline{hL}(?)_\sigma:
  \Com^-(X_{/S})\to \mathsf{HomLie}$ given by $\mathscr{A}\mapsto \mathscr{A}\cdot\mathbf{e}$, with
  product defined by (\ref{eq:homCom}). The collection of functors
  $\underline{hL}(?)_G:=\amalg_{\sigma \in G} \underline{hL}(?)_\sigma$ defines a
  functor $$\begin{array}{llll}
\underline{hL}(?)_G:&  \Com^G(X_{/S})&\longrightarrow &\mathsf{EquiHomLie}\\
&(\mathscr{A},G)&\longmapsto & \underline{hL}(\mathscr{A})_G.\end{array}$$
\end{thm}We will identify $\mathscr{A}\cdot\mathbf{e}$ and $\mathscr{A}$, saying that
  $\llangle\,\cdot,\cdot\,\rrangle_\sigma$ is a hom-Lie algebra on
  $A$. Notice also that $e\in G$ gives the abelian hom-Lie algebra. 
\begin{proof}Clearly, (\ref{eq:homCom}) defines a well-defined, skew-symmetric
  product. The only thing to check is the (twisted) Jacobi
  identity. The simplest (but tedious) proof of this is by direct 
  computation. Alternatively, one could use Theorem \ref{thm:twistprod} with 
  $\Delta_\sigma=\partial=(\id-\sigma)$. The statement concerning functoriality
  follows easily. 
\end{proof}
Notice that we avoid the slightly awkward assumptions on annihilators and $q$-commutativity of $\sigma$ and $\partial_\sigma$ needed for Theorem \ref{thm:twistprod} here. 

\begin{prop}\label{prop:isoCom}The hom-Lie algebras $\underline{hL}(\mathscr{A})_\sigma$ and
  $\mathscr{A}\cdot\partial$, for $\partial:=(\id-\sigma)$, are canonically isomorphic as hom-Lie algebras. 
\end{prop}
\begin{proof}The (obvious) isomorphism is $a\cdot\mathbf{e}\mapsto
  a\partial$. It is easy to check that this is indeed an isomorphism of hom-Lie algebras. 
\end{proof}
\section{Enveloping algebras}
For certain types of hom-Lie algebras we can associate a canonical associative algebra, analogous to enveloping algebras for Lie algebras. This holds in particular for hom-Lie algebras of twisted derivations on finitely generated algebras as we will now see. We do this in the affine case (globalizing should pose no problem).

Let $A$ be a (commutative) ring and $B$ a finitely generated (commutative) $A$-algebra 
$$A[y_1,y_2,\dots,y_n]\twoheadrightarrow B.$$ Let furthermore $\sigma$ be an $A$-algebra morphism and form $\Delta_\sigma:=\beta\cdot(\id-\sigma)\in \Der_{A,\sigma}(B)$, for some $\beta\in B$. Then Theorem \ref{thm:twistprod} endows $B\cdot\Delta_\sigma$ with the structure of a hom-Lie algebra. It is clear that the elements $$\mathbf{E}^{\underline{k}}:=y_1^{k_1}y_2^{k_2}\cdots y_n^{k_n}\cdot \Delta_\sigma, \qquad \underline{k}\in\mathbb{Z}_{\geq 0}^n$$ is a basis for $B\cdot\Delta_\sigma$ as an $A$-module. Then we have the relations
$$(\mathbf{E}^{\underline{k}})^\sigma\cdot\mathbf{E}^{\underline{l}}\,\,-\,\,
(\mathbf{E}^{\underline{l}})^\sigma\cdot \mathbf{E}^{\underline{k}}=\llangle\mathbf{E}^{\underline{k}},\mathbf{E}^{\underline{l}}\rrangle,$$ so we can form
$$\mathcal{E}(B\cdot \Delta_\sigma):=\frac{A\big\{\mathbf{E}^{\underline{k}}\mid \underline{k}\in \mathbb{Z}_{\geq 0}^n\big\}}{\Big((\mathbf{E}^{\underline{k}})^\sigma\cdot\mathbf{E}^{\underline{l}}\,-\,
(\mathbf{E}^{\underline{l}})^\sigma\cdot \mathbf{E}^{\underline{k}}\,-\,\llangle\mathbf{E}^{\underline{k}},\mathbf{E}^{\underline{l}}\rrangle\Big)}.$$Obviously this is in general a very complicated algebra because it is infinitely presented. Things simplify considerably if $B$ is finite as an $A$-module. 

So assume from now on that $B$ is a finite $A$-algebra:
$$B=Ay_1\oplus Ay_2\oplus\cdots\oplus Ay_n$$ and put
$$\eps_i:=y_i\cdot \Delta_\sigma.$$Then 
\begin{equation}\label{eq:env}
\mathcal{E}(B\cdot \Delta_\sigma):=\frac{A\big\{\eps_1,\eps_2,\dots,\eps_n\big\}}{\Big((\eps_i)^\sigma\cdot\eps_j-\,
(\eps_j)^\sigma\cdot\eps_i-\,\llangle\eps_i,\eps_j\rrangle\Big)}.
\end{equation}This was used in \cite{LaSi} to deform the Lie algebra $\mathfrak{sl}_2(k)$. Since it is illustrative and very helpful to have this construction in the back of ones mind in what comes, we will make a long story short and sketch that construction.
\subsection{Jackson $\mathfrak{sl}_2(k)$}
It is well-known that $\mathfrak{sl}_2(k)$ can be represented as differential operators on $k[t]$ as 
$$e=D_t,\quad h=-2tD_t, \quad  f=-t^2D_t.$$ Taking a $\sigma\in \mathrm{Aut}_k(k[t])$ we can ``deform'' this representation to generators
$$e=\Delta_\sigma,\quad h=-2t\Delta_\sigma,\quad f=-t^2\Delta_\sigma,$$for some $\sigma$-derivation $\Delta_\sigma$. Once again, Theorem \ref{thm:twistprod}, can be used to endow $k[e,h,f]$ with a hom-Lie algebra structure (under some restrictions on $\sigma$ to ensure closure of the bracket). For instance, suppose that $\sigma(t)=s_0+s_1t$ and $\Delta_\sigma(t)=1$, we can calculate the brackets to be 
$$\llangle h,e\rrangle = 2 e,\qquad \llangle h,f\rrangle = -s_0 h-2s_1 f, \qquad \llangle e,f\rrangle = -s_0 e + \frac{s_1+1}{2} h.$$ In \cite{LaSi} we called this algebra the \emph{Jackson $\mathfrak{sl}_2(k)$}, which we here denote by $\mathfrak{sl}_{2,J}(k)$.  Now, using the relation in (\ref{eq:env}) we can calculate 
\begin{align}\label{eq:slenv}
\begin{split}
	-2s_0 e^2+s_1he-eh & = 2e\\
	-2s_0ef+s_1hf+s_0^2eh-s_0s_1h^2-s_1^2fh & =  -s_0 h-2s_1 f\\
	ef+s_0^2e^2-s_0s_1he-s_1^2fe &= 	 -s_0 e + \frac{s_1+1}{2} h.
	\end{split}
	\end{align}Let me calculate the left-hand side of the first relation in  (\ref{eq:slenv}) to show the technique. 

According to Theorem \ref{thm:twistprod} what we need to compute is
$$\sigma(-2t)\Delta_\sigma(\Delta_\sigma)-\sigma(1)\Delta_{\sigma}(-2t\Delta_\sigma),$$since $e=\Delta_\sigma$ and $h=-2t\Delta_\sigma$. Expanding this we get 
\begin{multline*}
-2\sigma(t)\Delta_\sigma(\Delta_\sigma)+2\Delta_\sigma(-2t\Delta_\sigma)
=-2(s_0+s_1t)\Delta_\sigma(\Delta_\sigma)-\Delta_\sigma(-2t\Delta_\sigma)\\
=-2s_0\Delta_\sigma(\Delta_\sigma)-2s_1t\Delta_\sigma(\Delta_\sigma)-\Delta_\sigma(-2t\Delta_\sigma)=
-2s_0e^2+s_1he-eh.
\end{multline*}The other relations are calculated in the same way. 
	Relations (\ref{eq:slenv}) mean that 
	$$\mathcal{E}(\mathfrak{sl}_{2,J}(k))=\frac{k\{e,h,f\}}{\begin{pmatrix}
	2s_0 e^2+s_1he-eh -2e\\
	-2s_0ef+s_1hf+s_0^2eh-s_0s_1h^2-s_1^2fh +s_0 h+2s_1 f\\
	ef+s_0^2e^2-s_0s_1he-s_1^2fe +s_0 e - \frac{s_1+1}{2} h
	\end{pmatrix}}$$The most interesting case is perhaps when $s_0=0$, so $\sigma(t)=s_1t =qt$, then we get 
	$$\mathcal{E}(\mathfrak{sl}_{2,q}(k))=\frac{k\{e,h,f\}}{\begin{pmatrix}
	he-q^{-1}eh -2q^{-1}e\\
	hf-qfh +2q f\\
	ef-q^2fe  - \frac{q+1}{2} h
	\end{pmatrix}}$$This algebra has some remarkable properties as we will see later. Notice the similarity to the universal enveloping algebra of the ordinary $\mathfrak{sl}_2(k)$. For more details and much more see \cite{LaSi}. 
	
Notice that we really don't use that $\Delta_\sigma$ is a $\sigma$-derivation when computing the left-hand side of (\ref{eq:slenv}). Therefore, we can simplify computations by using the $\underline{hL}$-construction since this is equivalent to using $\Delta_\sigma=\id-\sigma$. 

\section{Arithmetic hom-Lie algebras}
\subsection{Definitions}
Suppose $\Lambda=\Lambda_{/\mathbb{Z}}$ is an excellent, regular, noetherian, integral domain. Then we say that $\Lambda$ is an \emph{arithmetic ring} (cf. \cite{GilletSoule}). A finitely generated, flat $\Lambda$-algebra $A$, for $\Lambda$ an arithmetic ring, is called an \emph{arithmetic algebra} (over $\Lambda$). 

Next, we define an \emph{arithmetic scheme} to be a separated, flat finite type morphism $X\to\underline{\Lambda}$ with $X$ integral and normal and $\Lambda$ an arithmetic ring. 
\begin{dfn}When $X$ is an arithmetic scheme, we refer to (equivariant) hom-Lie  algebras on $X$ as \emph{arithmetic {\rm(}equivariant{\rm )} hom-Lie algebras}.
\end{dfn}

\subsection{Corollaries}
The following corollary is an immediate consequence of theorem \ref{thm:invsheaf}.
\begin{corollary}\label{thm:TwistDerAritScheme}Let $X\to \underline{\Lambda}$ be an arithmetic scheme together with an action of a group $G$. Furthermore, let $\mathscr{E}$ be a $G$-equivariant, torsion-free sheaf on $X$. Then $\SDer_\sigma(\mathscr{E})$ is an invertible sheaf and
	$$G\setminus\{\id\}\to \mathrm{Pic}(X)\xrightarrow{\simeq} \mathrm{CDiv}(X),\quad \sigma \mapsto \SDer_\sigma(\mathscr{E})$$ associates to an element of $G$ an invertible sheaf and thus also an effective Cartier divisor. The image of $G$, together with the identity, generates a subgroup of $\mathrm{Pic}(X)\xrightarrow{\simeq} \mathrm{CDiv}(X)$.
\end{corollary} Once again, it would obviously be interesting to know what subgroup $G$ generates inside $\mathrm{Pic}(X)$.
\begin{corollary}\label{cor:AlgIntTwistDer}Let $L/K$ be a Galois extension of number fields and let $\mathcal{J}$ be a Galois stable fractional ideal. The morphism $\Spec(\mathfrak{o}_L)\to \Spec(\mathfrak{o}_K)$ together with $\mathcal{J}\in\mathrm{Pic}(\mathfrak{o}_L)$, satisfies the assumptions of the previous corollary. Thus, for every $\sigma\in\Gal(L/K)$, $\SDer_\sigma(\mathcal{J})$ is a fractional ideal. 
Furthermore, $\SDer_{\sigma}(\mathcal{J})$ is a Galois module for every $\sigma\in\Gal(L/K)$ and every fractional ideal $\mathcal{J}$. 
\end{corollary}
\begin{proof}Only the last statement requires proof. Take $\sigma\in\Gal(L/K)$. The $\mathfrak{o}_L$-module $\SDer_\sigma(M)$, where $M$ is a torsion-free  $\mathfrak{o}_L$-module (hence projective and locally free since $\mathfrak{o}_L$ is Dedekind), is locally free of rank one, locally given by generator of the form $\delta_{\sigma,\mathfrak{p}}:=a(\id-\sigma)$, $a\in \mathfrak{o}_{L,\mathfrak{p}}$. Extend the action of $\Gal(L/K)$ to $\delta_{\sigma,\mathfrak{p}}$ by $\tau\big(\delta_{\sigma,\mathfrak{p}}\big):=\tau(a)(\id-\sigma)$. Now extend the action of $\Gal(L/K)$ to $\SDer_\sigma(M)$ via the action on each localization as above. The result follows.
\end{proof}

\subsection{Families of equivariant (arithmetic) hom-Lie algebras}\label{sec:families}

The interesting thing about arithmetic schemes are the fibres. By definition, the fibre of $X\to\underline{\Lambda}$ at
$\mathfrak{p}\in \underline{\Lambda}$ is the
closed subscheme $$X_{\mathfrak{p}}:=f^{-1}(\mathfrak{p}):=X\times_{\underline{\Lambda}}
k(\mathfrak{p}),$$ where 
$k(\mathfrak{p})$ is the field of fractions of $\Lambda/\mathfrak{p}$. The fibre over the generic point $k((0))=\Frac(\Lambda)$ is the \emph{generic fibre}, all other are \emph{special fibres}. Notice that the generic fibre is a closed subscheme of $X$ over $k((0))$, and the special fibre over $\mathfrak{p}\in\underline{\Lambda}$ is a closed subscheme over $\Lambda/\mathfrak{p}$ (a finite field). From this follows that the closed (rational) points of $X$ come also come in two flavours, $L$-rational for $L$ an extension of $K$, on the generic fibre, and $\mathbb{F}$-rational where $\mathbb{F}$ is an extension of a finite field. 

For $\Fc$ a sheaf on $X$, flat over $\underline{\Lambda}$,
$\Fc\vert_{X_{\mathfrak{p}}}:=\iota^\ast\Fc$, where
$\iota:X_{\mathfrak{p}}\hookrightarrow X$ is the closed immersion, is a sheaf on
$X_{\mathfrak{p}}$. If $\Fc$ is a sheaf of $\Oc_X$-modules, then 
$\Fc\vert_{X_{\mathfrak{p}}}$ is a sheaf of
$\Oc_{{X_{\mathfrak{p}}}}$-modules. So, if $\Lc$ is an equivariant hom-Lie algebra over $\Ac$ on $X_{/\Lambda}$, then, by base
change, $\Lc\vert_{X_{\mathfrak{p}}}$ is an equivariant hom-Lie algebra over
$\Ac\vert_{X_{\mathfrak{p}}}:=\iota^\ast\Ac$ on $X_{\mathfrak{p}}$. In
this way, we get a flat family of equivariant hom-Lie algebras parametrized by
$\underline{\Lambda}$.

\begin{remark}Clearly, the above construction of families of equivariant hom-Lie algebras generalizes to a general $S$-scheme $X$.
\end{remark}

\subsection{Non-commutative arithmetic schemes}
We will now define a na\"ive notion  non-commutative arithmetic scheme. The reason for this is that it puts the results that follows in the proper context, even though the definition might seem a little meaningless. The definition we use is modelled upon the definition of non-commutative scheme given in \cite{Laudal}.

In the two-dimensional case, a different definition of non-commutative arithmetic scheme was given in \cite{Borek}, as a non-commutative generalization of Arakelov theory. 

\begin{dfn}
Let $A_{/\Lambda}$ be a non-commutative arithmetic algebra. Then we define the \emph{non-commutative arithmetic scheme} associated to $A$ to be
$$\ncspec(A):=(\Simp(A),\mathscr{O}_\pi^{\mathrm{ar}})=\big(\{\text{simple $A$-modules over $\Lambda$}\},\mathscr{O}_\pi^{\mathrm{ar}}\big).$$ The ``structure sheaf'' is defined
as $$\mathscr{O}_\pi^{\mathrm{ar}}:=\bigsqcup_{\mathfrak{p}\in\underline{\Lambda}} \mathscr{O}_\pi\big(A\otimes_{\Lambda} k(\mathfrak{p})\big),$$where $\mathscr{O}_\pi\big((A\otimes k(\mathfrak{p})\big)$ is the structure sheaf over the fibre of $A$ at $\mathfrak{p}$ defined in \cite{Laudal}. We define the \emph{coordinate ring} to be $A$. 
\end{dfn}
Notice that $A\otimes k(\mathfrak{p})$ is an algebra over a field, so the construction in \cite{Laudal} applies. 

\subsection{Equivariant cyclic hom-Lie algebras}\label{sec:classfield}

In this section we keep the following set of assumptions. We let $\Lambda$ be an arithmetic ring and $Y_{/\Lambda}$ be an arithmetic $\Lambda$-scheme and $\mathscr{A}$ be a coherent $\mathscr{O}_Y$-domain. Also, let $G_{/\Lambda}$ be an $\Lambda$-group scheme acting $\Lambda$-linearly and equivariantly on $\mathscr{A}$. By the results of Section \ref{sec:grouprep} we get an equivariant hom-Lie algebra on $\mathscr{A}$. Forming the global spectrum, $\mathbf{Spec}_Y(\mathscr{A})$, of $\mathscr{A}$ (which is an affine scheme over $Y$), we interpret this geometrically as an equivariant hom-Lie algebra on the rational points of $\mathbf{Spec}_Y(\mathscr{A})$. 

\subsubsection{$G$-covers}

Put $X:=\mathbf{Spec}_Y(\mathscr{A})$ and assume that $f : X_{/\Lambda}\to Y_{/\Lambda}$ is a (finite) $G$-cover, at most tamely ramified, with $X$ and $Y$ connected. Notice that this implies that $Y=X/G$ and that $X\to Y$ is \'etale over the complement of the branch locus. In addition, since $f$ is finite, $\mathscr{A}$ is a locally free sheaf of finite rank. Take $\sigma\in G$ and consider $\SDer_\sigma(\mathscr{A})$. This is an invertible sheaf over $X\setminus \mathtt{Ram}(f)$ which can be extended to an invertible sheaf on the whole $X$ if $X$ is regular. Inside $\SDer_\sigma(\mathscr{A})$ we consider the submodule $\mathscr{A}\cdot\Delta_\sigma$, with $\Delta_\sigma:=\id-\sigma$. 
\begin{remark}The reason for considering $\mathscr{A}\cdot\Delta_\sigma$ and not the whole $\SDer_\sigma(\mathscr{A})$ is simply a matter of simplifying the calculations and expressions. What follows can be done for $\SDer_\sigma(\mathscr{A})$ in exactly the same way. 
\end{remark}
Note that $$\mathscr{A}\cdot\Delta_\sigma(U)=\bigoplus_{i=0}^n \mathscr{O}_Y(U) e_i\Delta_\sigma=\bigoplus_{i=0} \mathscr{O}_Y(U)\eps_i,\qquad U\subseteq Y,$$ with $\eps_i:=e_i\Delta_\sigma$, and $e_i$ generating sections over $U$. We consider the hom-Lie algebra $(\mathscr{A}\cdot\Delta_\sigma, \llangle\,\,,\,\,\rrangle)$. By definition we see that $(\mathscr{A}\cdot\Delta_\sigma, \llangle\,\,,\,\,\rrangle)=\underline{hL}(\mathscr{A})_\sigma$. 
\subsubsection{Witt hom-Lie algebras}
\begin{prop}\label{prop:general_KummerWitt}Let $\mathscr{A}$ be a finite $\mathscr{O}_Y$-algebra generated (locally) by sections $e_0, e_1, \dots, e_{n-1}$ and with structure constants
$$e_ie_j=\sum_{k=0}^{n-1}a_{ij}^ke_k.$$Suppose further that $\sigma\in G$ acts as $\sigma(e_i)=\sum_{k=0}^{n-1} s_{ik}e_k$ with $s_{ik}\in\mathscr{O}_Y$. Then 
$$\W^\mathscr{A}_\sigma:=(\mathscr{A}\cdot \Delta_\sigma, \llangle\,\,,\,\,\rrangle)$$ is given by 
$$\llangle \eps_i,\eps_j\rrangle = \sum_{\ell=0}^{n-1}\Big(\sum_{k=0}^{n-1}\big(s_{ik}a^\ell_{kj}-s_{jk}a^\ell_{ki}\big)\Big)\eps_\ell.$$
\end{prop}
\begin{proof}
Simple computation. 
\end{proof}
Notice the special case when $\sigma(e_i)=q_ie_i$, with $q_i\in \mathscr{O}_Y^\times$:
\begin{equation}\label{eq:arithmeticWitt}
\llangle \eps_i,\eps_j\rrangle = \sum_{k=0}^{n-1} (q_i-q_j)a^k_{ij}\eps_k.
\end{equation} 

We call $\W^\mathscr{A}_\sigma$ the (generalized) \emph{Witt hom-Lie algebra} (over $\mathscr{O}_Y$) associated with $\sigma$ and $\mathscr{A}$. 

\begin{remark}
Observe the abuse of notation: we write $\mathscr{A}$ as an affine algebra, when, strictly speaking, it should be given as \emph{sheaf} of affine algebras. In other words,
$$\mathscr{A}=\bigoplus_{i=0}^{n-1}\mathscr{L}^{\otimes i}.$$We will be sloppy on this point in what follows in order to avoid drowning in heavy notation. The underlying meaning should be clear, however. 
\end{remark}
\subsubsection{Kummer--Witt hom-Lie algebras}
In this section we study the simplest family of examples of $G$-covers, namely, cyclic covers. In this case 
$$\mathscr{A}=\mathscr{O}_Y[t]/(t^n-b)=\bigoplus_{i=0}^{n-1}\mathscr{O}_Y e_i, \quad e_i:=t^i, $$ for a global section $b\in \mathscr{O}_Y$. We assume that $\mathscr{O}_Y$ includes the $n$-th roots of unity. In fact, $\mathbf{Spec}_Y(\mathscr{A})$ is a cyclic cover of $Y$ with $\sigma(t):=\xi^r t$, $0\leq r\leq n-1$, for $\xi$ a primitive $n$-th root of unity. 

Observe that we allow $b=0$ in which case we view $\mathscr{A}$ as an ``infinitesimal thickening'' of $Y$. 

 Put $\eps_i:=t^i\Delta_\sigma$. Then it is easy to prove
\begin{corollary}\label{prop:tame_algebra}The hom-Lie algebra structure on $\mathscr{A}\cdot\Delta_\sigma$ is given by
\begin{align*}
\llangle \eps_i, \eps_j\rrangle =b^\circlearrowright\xi^i(1-\xi^{j-i})\eps_{\{i+j\!\!\!\mod n\}}, \quad i\leq j,
\end{align*}where $b^\circlearrowright$ means that $b$ is included when $i+j\geq n$. 
\end{corollary}
\begin{proof}Follows immediately from Proposition \ref{prop:general_KummerWitt}.
\end{proof}We denote the locally free algebra in the proposition by $\KW_{b,\sigma}(\xi)=\KW^\mathscr{A}_{b,\sigma}(\xi)$ and refer to it as a \emph{ Kummer--Witt hom-Lie algebra}. We often suppress the dependence on $\sigma$ in the notation. 

Here comes a few illustrative examples. 
\begin{example}\label{exam:n=3}We look first at the example when $n=3$ and $\sigma(t)=\xi t$, $\xi^3=1$. Putting this into the structure-constant-machine in the above corollary gives
\begin{align}
\begin{split}
  \llangle\eps_0,\eps_1\rrangle &= (1-\xi)\eps_1\\
  \llangle\eps_0,\eps_2\rrangle &= (1-\xi^2)\eps_2\\
  \llangle\eps_1,\eps_2\rrangle &= b\xi(1-\xi)\eps_0.
\end{split}
\end{align}Instead taking $\sigma(t)=\xi^2 t$ gives
\begin{align}
\begin{split}
  \llangle\eps_0,\eps_1\rrangle &= (1-\xi^2)\eps_1\\
  \llangle\eps_0,\eps_2\rrangle &= (1-\xi)\eps_2\\
  \llangle\eps_1,\eps_2\rrangle &= -b\xi(1-\xi)\eps_0.
\end{split}
\end{align}Obviously, the case when $\sigma$ is the identity gives the abelian hom-Lie algebra. Notice that the three algebras in the equivariant structure are non-isomorphic. 
\end{example}
\begin{example}\label{exam:n=4}Now we study the case $n=4$ and we begin with $\sigma(t)=\xi t$. We get 
\begin{align}\label{eq:exam_n=4_1}
	\begin{split}
 	\llangle\eps_0,\eps_1\rrangle &= (1-\xi)\eps_1\\
  	\llangle\eps_0,\eps_2\rrangle &= (1-\xi^2)\eps_2=2\eps_2\\
  	\llangle\eps_0,\eps_3\rrangle &= (1-\xi^3)\eps_3\\ 
  	\llangle\eps_1,\eps_2\rrangle &= \xi(1-\xi)\eps_3\\
  	\llangle \eps_1,\eps_3\rrangle &=b\xi(1-\xi^2)\eps_0=2b\xi\eps_0 \\
  	\llangle \eps_2,\eps_3\rrangle &=b\xi^2(1-\xi)\eps_1=-2b(1-\xi)\eps_1,
	\end{split}
\end{align}where we have used that $\xi^2=-1$ for $\xi$ a fourth root of unity. Clearly this is rather similar in structure to the case $n=3$ (but we will see shortly that this is a mirage; the case $n=3$ is very special). However, when $\sigma(t)=\xi^2t$ a more surprising structure emerges:
\begin{align}\label{eq:exam_n=4_2}
	\begin{split}
 	\llangle\eps_0,\eps_1\rrangle &= 2\eps_1\\
  	\llangle\eps_0,\eps_2\rrangle &= 0\\
  	\llangle\eps_0,\eps_3\rrangle &= 2\eps_3\\
  	\llangle\eps_1,\eps_2\rrangle &= -2\eps_3\\
  	\llangle \eps_1,\eps_3\rrangle &=0\\
  	\llangle \eps_2,\eps_3\rrangle &=2b\eps_1,
	\end{split}
\end{align} It is rather easy to see that this algebra is solvable. The case $\sigma(t)=\xi^3 t$ is similar to (\ref{eq:exam_n=4_1}).
\end{example}The last example is when $n=5$. 
\begin{example}\label{exam:n=5}We take $\sigma(t)=\xi t$ and get
\begin{align*}
 	\llangle\eps_0,\eps_1\rrangle &= (1-\xi)\eps_1 &
  	\llangle\eps_0,\eps_2\rrangle &= (1-\xi^2)\eps_2 & 
  	\llangle\eps_0,\eps_3\rrangle &= (1-\xi^3)\eps_3  \\
  	\llangle \eps_0,\eps_4\rrangle & =(1-\xi^4)\eps_4 &  
  	\llangle\eps_1,\eps_2\rrangle &= \xi(1-\xi)\eps_3 &
  	\llangle \eps_1,\eps_3\rrangle &=\xi(1-\xi^2)\eps_4 \\
  	\llangle \eps_1,\eps_4\rrangle &=b\xi(1-\xi^3)\eps_0 & 
  	\llangle \eps_2,\eps_3\rrangle &=b\xi^2(1-\xi)\eps_0 &
  	\llangle \eps_2,\eps_4\rrangle & =b\xi^2(1-\xi^2)\eps_1 \\
  	\llangle \eps_3,\eps_4\rrangle &=b\xi^3(1-\xi)\eps_2. 
\end{align*}Clearly, this is also quite similar in structure to Example \ref{exam:n=3} (but this is also a mirage). The other $\sigma\in G$ give similar structure constants. Since $n=5$ is a prime nothing interesting as in (\ref{eq:exam_n=4_2}) happens (as the reader can easily conclude). 
\end{example}

For all $\sigma\in G$ we have the following subalgebra:
\begin{prop}\label{prop:subalgebra}
	The algebra $\Jac_b(\xi)=\Jac^\mathscr{A}_{b,\sigma}(\xi)$ given by
	\begin{align}
\begin{split}
  \llangle\eps_0,\eps_1\rrangle &= (1-\xi)\eps_1\\
  \llangle\eps_0,\eps_{n-1}\rrangle &= (1-\xi^{n-1})\eps_{n-1}\\
  \llangle\eps_1,\eps_{n-1}\rrangle &= b\xi(1-\xi^{n-2})\eps_0
\end{split}
\end{align}is a subalgebra of $\KW^\mathscr{A}_{b,\sigma}(\xi)$. Furthermore, if $b\neq 0$, it is non-solvable if $n=p>2$ is a prime. If $n=2$, $\Jac_b(\xi)$ is clearly solvable. 
\end{prop}
\begin{proof}
	The first statement follows from Proposition \ref{prop:tame_algebra}, whereas the second follows from $\llangle \Jac_b(\xi),\Jac_b(\xi)\rrangle=\Jac_b(\xi)$ and induction.
\end{proof}
Notice the similarity between $\Jac_b(\xi)$ and the the Jackson-$\mathfrak{sl}_2$ from before. It is therefore natural to call the algebra $\Jac_b(\xi)$ the \emph{Jackson subalgebra} of $\KW^\mathscr{A}_{b,\sigma}$.
\begin{conj}If $n$ is composite then there is at least one $\sigma\in G$ such that $\KW_{b,\sigma}^\mathscr{A}(\xi)$ is solvable.
\end{conj}
In the cases I've investigated this seems to be true and the following proposition gives some support for this claim.
\begin{prop}
Let $n$ be composite. Then for some $\sigma\in G$ there are $0\leq i\neq j\leq n-1$ such that 
$$\llangle \eps_i,\eps_j\rrangle = 0.$$ 
\end{prop}
\begin{proof}
	Let $\xi$ be a primitive $n$-th root of unity. Since $n$ is composite there is a $k<n$ such that $k\vert n$. Consider $\sigma(t)=\xi^k t$. Then there are $i< j<n$ such that $k(j-i)=n$. The claim follows. 
\end{proof}The problem now reduces to the question whether this is sufficient for solvability. In any case, it is clear that the equivariant hom-Lie algebra structures are richer when $e$ is composite. For example, there are more ``obvious'' subalgebras and the following loose feeling seems reasonable: 
\begin{feeling}When $n$ is composite, the number of subalgebras are more than in the case when $n$ is prime.
\end{feeling}For instance, in the prime examples above the only ``obvious'' subalgebras are the ones described in Proposition \ref{prop:subalgebra} (we do not include the ``extremely obvious ones'' on the form $\llangle \eps_0,\eps_i\rrangle$). On the other hand, the algebra given by (\ref{eq:exam_n=4_2}) has at least one more subalgebra.

I have no idea how to make this feeling precise enough to warrant calling it a ``Conjecture''. 
\subsubsection{Artin--Schreier divisors}
When the characteristic is $p>0$, cyclic extensions look like
$$A=B[y]/(y^p-y-b),\quad b\in B.$$ These are called \emph{Artin--Schreier extensions}. The action of the Galois group $\mathbb{Z}/p$ is given by $\sigma(y)=y+\nu$, $0\leq \nu\leq p-1$. We can easily compute the associated hom-Lie algebras as before. For obvious reasons we call these \emph{Artin--Schreier hom-Lie algebras}.
\begin{example}
Let $p=3$:
\begin{align*}
	\llangle\eps_0,\eps_1\rrangle & = -\nu\eps_0\\
	\llangle\eps_0,\eps_2\rrangle & = -2\nu\eps_1-\nu^2\eps_0\\
	\llangle\eps_1,\eps_2\rrangle & = -\nu^2\eps_1-\nu\eps_2.
\end{align*}Notice that this algebra is not dependent on the divisor $b$. 
\end{example}
\begin{example}
	For $p=5$ we get:
\begin{align*}
	\llangle\eps_0,\eps_1\rrangle & = -\nu\eps_0 \\
	\llangle\eps_0,\eps_2\rrangle & = -2\nu\eps_1-\nu^2\eps_0 \\
	\llangle\eps_0,\eps_3\rrangle & =  -3\nu\eps_2-3\nu^2\eps_1-\nu^3\eps_0 \\
	\llangle \eps_0,\eps_4\rrangle & = -4\nu\eps_3-\nu^2\eps_2-4\nu^3\eps_1-\nu^4\eps_0\\
	\llangle \eps_1,\eps_2\rrangle & = -\nu\eps_2-\nu^2\eps_1 \\
	\llangle \eps_1,\eps_3\rrangle & = -2\nu\eps_3-3\nu^2\eps_2-\nu^3\eps_1\\
	\llangle \eps_1,\eps_4\rrangle & = -3\nu\eps_3-\nu^2\eps_3-4\nu^3\eps_2-\nu^4\eps_1\\
	\llangle \eps_2,\eps_3\rrangle & = -\nu\eps_4-2\nu^2\eps_3-\nu^3\eps_2\\
	\llangle\eps_2,\eps_4\rrangle & = -4\nu^3\eps_3-\nu^3\eps_2-2\nu \eps_1-2\nu b\eps_0\\
	\llangle\eps_3,\eps_4\rrangle & = -3\nu^3\eps_4-\nu^4\eps_3-\nu\eps_2-(b+3\nu)\nu\eps_1-3b\nu^2\eps_0
\end{align*}
	It is relatively easy to deduce general formulas for the products $\llangle \eps_0,\eps_i\rrangle$ and $\llangle \eps_1,\eps_i\rrangle$, but I haven't found any easy way to write them in general. 
\end{example}	
	Clearly we have more ``obvious subalgebras'' compared to the previous Kummer case. In this case I have neither a conjecture nor a feeling other than:
	\begin{feeling}There are ``many'' subalgebras for Artin--Schreier hom-Lie algebras.
\end{feeling}
From the examples above we can also see that the complexity is much greater for Artin--Schreier hom-Lie algebras as compared to the Kummer case.

\subsubsection{Ramified divisors}
The above can be used directly to study the ramification locus of a $G$-cover of arithmetic schemes $X\to Y$. 

By the Zariski--Nagata purity theorem the ramification locus of $f:X\to Y$ is concentrated in codimension one so any ramification occurs along a divisor $B\subset Y$. Let $B_i$ be the irreducible components of $B$ and $\eta_i$ their generic points. We have the natural extension of local rings 
$$\mathscr{O}_{Y,\eta_i}\hookrightarrow \mathscr{O}_{X,\eta_i}\quad \text{in} \quad K(X)/K(Y),$$where $\mathscr{O}_{X,\eta_i}$ is the integral closure of $\mathscr{O}_{Y,\eta_i}$ in $K(X)$.  By definition, ramification along $B$ translates into the ramification properties of the extension $K(X)/K(Y)$. 

For simplicity we shall assume that $B$ is connected (so $i=1$) and we write  $\mathscr{B}:=\mathscr{O}_{Y,\eta}$ (branch locus) and $\mathscr{R}:=\mathscr{O}_{X,\eta}$ (ramification locus). The field extension 
$$\Frac(\widehat{\mathscr{R}})/\Frac(\widehat{\mathscr{B}})$$ is an extension of local fields and $\widehat{\mathscr{R}}$ and $\widehat{\mathscr{B}}$ are discrete valuation rings. The decomposition group $D\subseteq G$ at $\eta$ is the Galois group of $\Frac(\widehat{\mathscr{R}})/\Frac(\widehat{\mathscr{B}})$. We will study the ramification properties of this extension. 

By an \'etale base change we can assume that $X\to Y$ is totally ramified at $\Bhat$ (pass to the maximal unramified subextension of $\Rhat/\Bhat$). Let $\pi_{\Bhat}$ and $\pi_{\Rhat}$ be uniformizers of $\Bhat$ and $\Rhat$, respectively. We have that $\pi_{\Bhat}=\pi_{\Rhat}^e$, where $e$ is the ramification index of $\Rhat/\Bhat$. Since we assume that the extension is totally ramified we have that 
$$e=\big[\Frac(\Rhat)\big/\Frac(\Bhat)\big]=\# D.$$

By the definition of tame ramification the extension
$$\Frac\big(\Rhat/(\pi_{\Rhat})\big)\Big/\Frac\big(\Bhat/(\pi_{\Bhat})\big)$$ is separable so the ring extension $\Rhat/\Bhat$ is monogenic, and can in fact be explicitly given as
$$\Rhat=\Bhat[t]\big/(t^e-\pi_{\Bhat})$$ (e.g., \cite[Proposition 3.5, Chapter 3]{FesenkoVostokov}). Obviously this is a cyclic extension so the discussion in the previous subsection applies by localizing the sheaves $\mathscr{A}$ and $\mathscr{O}_Y$ and then completing. We can therefore use Proposition \ref{prop:tame_algebra} the explicitly give to structure of the ramification as a hom-Lie structure. The examples and conjectures following from the Proposition \ref{prop:tame_algebra} applies equally in the present case.

\subsection{Non-commutative arithmetic schemes attached to $G$-covers}
We will now construct non-commutative arithmetic schemes associated to the ramification divisors by constructing the enveloping algebras to the hom-Lie algebras coming from the ramification structure.  

We begin by recalling the following definitions.

Let $R$ be a commutative ring, $A$ an $R$-algebra and $M$ a finitely generated $A$-module. Then  $M$ is called \emph{generically free} if there is a non-zero divisor $s\in A$ such $M[s^{-1}]:=M\otimes_A A[s^{-1}]$ is free. An algebra $A$ is \emph{strongly Noetherian} if for every, not necessarily a commutative, $R$-algebra $R'$, $A\otimes_R R'$ is Noetherian. 

For completeness sake we also include the following definition.
\begin{dfn}
\begin{itemize}
	\item[(i)] Let $R$ be a ring and $M$ an $R$-module. Then the \emph{grade} of $M$ is defined as
$$j(M):=\mathrm{min}\{i\mid \mathrm{Ext}_R^i(M,R)\neq 0\}$$ or $j(M)=\infty$ if no such $i$ exists. 
	\item[(ii)] $R$ is \emph{Auslander--Gorenstein} if for every left and right Noetherian $R$-module $M$ and for all $i\geq 0$ and all $R$-submodules $N\subseteq \mathrm{Ext}_R^i(M,R)$, we have $j(N)\geq i$.
	\item[(iii)] $R$ is \emph{Auslander-regular} if it is Auslander--Gorenstein and has finite global dimension. 
	\item[(iv)] Let $R$ be a $K$-algebra, for $K$ a field. Then $R$ is \emph{Cohen--Macaulay} (CM) if $$j(M)+\mathrm{GKdim}(M)<\infty,$$ for every $R$-module $M$. Here $\mathrm{GKdim}$ denotes Gelfand--Kirillov dimension with respect to $K$.
\end{itemize}
\end{dfn}
\subsubsection{Non-commutative arithmetic Witt and Kummer--Witt schemes}
	For simplicity we will for this section assume that $Y=\Spec(B)$ for $B$ a commutative $\Lambda$-algebra. Presumably everything that follows globalizes without too much effort. 

Let $A$ be given as the $B$-module generated by $e_0, e_1, \dots, e_{n-1}$ with structure constants
$$e_ie_j=\sum_{k=0}^{n-1} a^k_{ij}e_k, \quad a^k_{ij}\in B$$ and let $\sigma\in G\subseteq\Aut_B(A)$ act on $e_i$ as $\sigma(e_i)=q_i e_i$, $q_i\in B^\times$. Then $\W^A$ is given as in (\ref{eq:arithmeticWitt}) and we can form
$$\eW_\sigma:=\mathcal{E}(\W^A)=\frac{B\{\eps_0,\eps_1,\dots, \eps_{n-1}\}}{\big(\eps_i\eps_j-q^{-1}_iq_j\eps_j\eps_i-\sum_{k=0}^{n-1}(1-q_i^{-1}q_j)a^k_{ij}\eps_k\big)}.$$The associated non-commutative arithmetic scheme is
$$\underline{\eW}_\sigma:=\ncspec\Big(\frac{B\{\eps_0,\eps_1,\dots, \eps_{n-1}\}}{\big(\eps_i\eps_j-q^{-1}_iq_j\eps_j\eps_i-\sum_{k=0}^{n-1}(1-q_i^{-1}q_j)a^k_{ij}\eps_k\big)}\Big),$$ the \emph{non-commutative Witt scheme}. This is obviously not a generalization of the commutative Witt scheme. I feel that it is nevertheless an appropriate name for this object. 
\begin{prop}Assume that $B$ is a regular algebra. Then $\eW_\sigma$ is Auslander-regular. 
\end{prop}
\begin{proof}Notice that the monomials in $\eps_0, \eps_1,\dots,\eps_{n-1}$ forms a basis for $\eW_\sigma$ as a $B$-module and the relations between the $\eps_i$'s are on the form 
$\eps_i\eps_j-q_{ij}\eps_j\eps_i = \sum_{k=0}^{n-1} c a^k_{ij}\eps_k$. Furthermore, for $b\in B$, we have that $\eps_i b=b\eps_i$, for all $i$. Then \cite[Theorem 1 and Corollary 2]{GomezLobillo} implies that $\eW_\sigma$ is Auslander-regular. 
\end{proof}
We now form the Ore extension $S:=B[y_0;\sigma_0][y_1;\sigma_1]\cdots[y_{n-1};\sigma_{n-1}]$ with $\sigma_i(y_j)=q_{ij}y_j$, with $q_{ij}:=q^{-1}_iq_j$. Notice that $S$ is in fact $\mathrm{gr}(\eW_{\sigma})$ with $\eW_{\sigma}$ given the standard filtration. 
\begin{prop}\label{prop:Witt_PI}Assume that the $q_i$'s are $m_i$-th roots of unity. Then the centre of $\mathrm{gr}(\eW_\sigma)$ is 
$$\zeta(S)=\zeta\big(\mathrm{gr}(\eW_\sigma)\big)=B[y_0^N, y_1^N,\dots, y_{n-1}^N],$$ where $N$ is the least common multiple of the $m_i$. The algebra $S=\mathrm{gr}(\eW_\sigma)$ is finite as a module over its centre and hence a polynomial identity (PI) algebra. If, in addition, $S$ is prime, then it is an order in its quotient ring of fractions. 
\end{prop}
\begin{proof}It is straightforward to check that the centre is as claimed, and from this follows that $S$ is PI by \cite[13.1.13]{McConnellRobson}. Indeed, we have
$$\eps_i\eps_{j}^k=q_{ij}^{k}\eps_{j}^k\eps_i,\quad 0\leq i<j\leq n-1, \,\,\,k\in \N.$$ Clearly from this follows that $$\eps_i\eps_j^{m_i}=q_{ij}^{m_i}\eps_j^{m_i}\eps_i \iff \eps_i\eps_j^{m_i}=\eps_j^{m_i}\eps_i.$$ If $N$ is the least common multiple of the $m_i$ then $q_{ij}^N=1$, for all $0\leq i<j\leq n-1$. From this the first two claims follow. The third claim, that $S$ is an order in its quotient ring of fractions, follows from \cite[13.6.6]{McConnellRobson}.
\end{proof}
The fact that $\mathrm{gr}(\eW_\sigma)=S$ has the following consequence.
\begin{prop}Let $B$ be a domain (noetherian). Then the $B$-algebra $\eW_\sigma$ is a domain (noetherian). 
\end{prop}
\begin{proof}
	This follows from the fact that the standard filtration is separated, which implies that $\mathrm{gr}(\eW_\sigma)$ is a domain (noetherian) if and only if $\eW_\sigma$ is a domain (noetherian). Since $\mathrm{gr}(\eW_\sigma)$ is an iterated Ore extension of a domain (noetherian ring) it is itself a domain (noetherian ring). 
\end{proof}An \emph{admissible} (commutative) ring (or scheme) is a ring which is of finite type  over a field or excellent Dedekind domain. 
\begin{prop}Let $B$ be an admissible domain and endow $\eW_\sigma$ with the standard filtration (with generators in degree one). Then every finite $\eW_\sigma$-module is generically free.
\end{prop}
\begin{proof}
	We have seen that $\mathrm{gr}(\eW_\sigma)$ is an Ore extension and by \cite[Proposition 4.4]{ArtinSmallZhang}, $\mathrm{gr}(\eW_\sigma)$ is strongly noetherian. We can now apply \cite[Theorem 0.3]{ArtinSmallZhang} to conclude. 
\end{proof}
\begin{prop}Suppose $\eW_\sigma$ is Auslander-regular. Then $K_0(\eW_\sigma)\approx K_0(B)$, where $K_0(R)$ is the Grothendieck group of projective $R$-modules. 
\end{prop}
\begin{proof}
	Since $\eW_\sigma$ is Auslander-regular the global dimension is finite. This implies that every cyclic $\eW_\sigma$-module has finite projective dimension (e.g., \cite[7.1.8]{McConnellRobson}). From this the result now follows from Quillen's theorem \cite[Theorem 12.6.13]{McConnellRobson}.
\end{proof}
Notice that $K_0(\mathfrak{o}_K)=\mathrm{Pic}(\mathfrak{o}_K)\oplus \Z$, the isomorphism given by the Chern character:
$$\mathrm{ch}:\,\,\, K_0(\mathfrak{o}_K)\longrightarrow \mathrm{Pic}(\mathfrak{o}_K)\oplus \Z,\quad E\mapsto \det(E)\oplus \mathrm{rk}(E).$$ Hence, if $B=\mathfrak{o}_K$, we get 
\begin{corollary}When $B=\mathfrak{o}_K$, we have $K_0(\eW_\sigma)\approx\mathrm{Pic}(\mathfrak{o}_K)\oplus \Z$. 
\end{corollary}

Put $\eW_{\sigma,\pp}:=\bar{\eW}_\sigma \otimes_\Lambda k(\pp)$. Assume that $\bar{B}_{\pp}:=B\otimes_\Lambda k(\pp)$ is affine and filtered by the standard filtration with generators $b_1, b_2, \dots, b_m\in \mathrm{Fil^1}\bar{B}_{\pp}$. Assume also that $\bar q_i\neq 0$, i.e., that $q_i\notin \pp$. 

We now form the Ore extension $\bar{S}_\pp:=\bar{B}_{\pp}[\bar y_0;\bar \sigma_0][\bar y_1;\bar \sigma_1]\cdots[\bar y_{n-1};\bar\sigma_{n-1}]$ with $\bar\sigma_i(\bar y_j)=\bar q_{ij}\bar y_j$, with $\bar q_{ij}:=\bar q^{-1}_i\bar q_j$. Notice that $S$ is in fact $\mathrm{gr}(\bar{\eW}_{\sigma,\pp})$ with $\bar{\eW}_{\sigma,\pp}$ given the standard filtration. 
\begin{prop}Assume that either $\mathrm{gr}^{\mathrm{Fil}}(\bar{B}_{\pp})$ or $\bar S_\pp$ is Cohen--Macaulay. Then $\bar{\eW}_{\sigma,\pp}$ is Cohen--Macaulay. In addition, $\bar S_\pp$ is a maximal order in its quotient ring of fractions. 
\end{prop}
\begin{proof}
The assumptions imply that we can apply \cite[Theorem 3]{GomezLobillo}. The last statement follows from \cite{Stafford}. 
\end{proof}

	We now specialize the above construction. Let $A$ be the finite $B$-algebra 
	$$A:=B[t]/(t^n-b)=Be_0\oplus Be_1\oplus Be_2\oplus\cdots\oplus Be_{n-1},$$with $e_i:=t^i$ and $b\in B$.   Also, we recall the assumption that $\xi\in\Lambda$. 
	
From Proposition \ref{prop:tame_algebra} we see that 
	$$\eKW_b(\xi):=\mathcal{E}(\KW_b^A)=\frac{B\{\eps_0,\eps_1,\dots,\eps_{n-1}\}}{\big(
	\eps_i\eps_j-\xi^{j-i}\eps_j\eps_i-b^\circlearrowright(1-\xi^{j-i})\eps_{\{i+j\!\!\!\mod n\}}\big)}, \quad b\in B,$$and so, forming the non-commutative spectrum 
	\begin{align*}
	\underline{\eKW_b(\xi)}&:=\ncspec\big(\eKW_b(\xi)\big)\\
	&=\ncspec\Big(\frac{B\{\eps_0,\eps_1,\dots,\eps_{n-1}\}}{\big(\eps_i\eps_j-\xi^{j-i}\eps_j\eps_i-b^\circlearrowright(1-\xi^{j-i})\eps_{\{i+j\!\!\!\mod n\}}\big)}\Big),\end{align*}we get a non-commutative arithmetic scheme over $\Lambda$. It is reasonable to call this the \emph{non-commutative arithmetic Kummer-Witt scheme attached to the cover}. 
	
We will now study a canonical subalgebra of $\eKW_b(\xi)$ in some detail and show that it has some remarkable properties. First, notice that since $\xi^n=1$, we have that $\xi^{-(n-1)}=\xi$ and $\xi^{-(n-2)}=\xi^2$. We put
$$\eJac_b(\xi):=\mathcal{E}\big(\Jac_{b,\sigma}^A)\big)=
\frac{B\{\eps_0,\eps_1,\eps_{n-1}\}}{\begin{pmatrix}
 \eps_0\eps_1-\xi\eps_1\eps_0- (1-\xi)\eps_1\\
  \eps_{n-1}\eps_0-\xi\eps_0\eps_{n-1}- (1-\xi)\eps_{n-1}\\
  \eps_{n-1}\eps_1-\xi^2\eps_1\eps_{n-1}- b(1-\xi^2)\eps_0
\end{pmatrix}}.$$
\begin{prop}
	The algebra $\eJac_b(\xi)$ is isomorphic to an iterated Ore extensions and so have a Poincar\'e--Birkhoff--Witt basis.
\end{prop}
\begin{proof}
By changing basis $\eps_0\to (1-\xi^2)^{-1}\eps_0+1$ we can transform the relations to 
 \begin{align}\label{eq:envS}
 \begin{split}
\eps_0\eps_1 -\xi\eps_1\eps_0 &= 0, \\
 \eps_{n-1}\eps_0-\xi\eps_0\eps_{n-1} & = 0, \\ 
 \eps_{n-1}\eps_1-\xi^2\eps_1\eps_{n-1} & =b\eps_0+b(1-\xi^2).
 \end{split}
 \end{align}From this we construct the iterated Ore extension
 $$B[\eps_1][\eps_0;\tau][\eps_{n-1}; \omega,\Delta],$$with 
 \begin{align*}
 \tau(\eps_1)&=\xi\eps_1, & \omega(\eps_1)&=\xi^2\eps_1,&
 \Delta(\eps_1) & =b\eps_0+b(1-\xi^2),\\ 
 \omega(\eps_0)& =\xi\eps_0,&\Delta(\eps_0)&=0.
 \end{align*}
 To ensure that this is well-defined we need to verify that $\Delta$ respects the relation $\eps_0\eps_1-\xi\eps_1\eps_0=0$. This is easy and is left for the reader. The last statement follows since $\eps_0,\eps_1,\eps_2$ are linearly independent over $B$. 
 \end{proof}
 Since $  \eJac_b(\xi)$ and the algebra defined by (\ref{eq:envS}) are isomorphic, and the relations (\ref{eq:envS}) are simpler, we naturally work with this algebra instead. By abuse of notation we denote this also by $  \eJac_b(\xi)$. 
 \begin{remark}The algebras $\eKW_b(\xi)$ are \emph{not} Ore extensions in general. 
 \end{remark}

\begin{prop}The algebra $\eJac_b(\xi)$ is finite as a module over its centre $$\zeta\big(  \eJac_b(\xi)\big)=B[\eps_0^n,\eps_1^n,\eps_{n-1}^n]$$ and hence a polynomial identity (PI), if in addition $\eJac_b(\xi)$ is prime it is an order in its quotient ring of fractions. If $\eJac_b(\xi)$ is a $K$-algebra, $K$ a field, then it is even a maximal order.
\end{prop}
\begin{proof}The proof is essentially the same as the proof of Proposition \ref{prop:Witt_PI}.  Using induction one can prove that
$$\eps_1\eps_{n-1}^k=\xi^{-2k}\eps_{n-1}^k\eps_1-
b\xi^{-2k}\eps_{n-1}^{k-1}\big(\eps_0\sum_{i=0}^{k-1}\xi^{i}+(1-\xi^2)\sum_{i=0}^{k-1}\xi^{2i}\big).$$Since $\xi$ is an $n$-th root of unity, we see that for $k=n$ the sums in the right-hand side vanishes and we are left with $\eps_1\eps_{n-1}^{n}=\eps_{n-1}^n\eps_1$. Similarly (by induction),
$$\eps_{n-1}\eps_1^k=\xi^{2k}\eps_1^k\eps_{n-1}+b\xi^{k-1}\sum_{i=0}^{k-1}\xi^i \eps_1^{k-1}\eps_0+b(1-\xi^2)\sum_{i=0}^{k-1}\xi^{2i}\eps_1^{k-1}.$$ From this follows that $\eps_{n-1}\eps_1^n=\eps_1^n\eps_{n-1}$. That $\eps_1^n$ commutes with $\eps_0$ and that $\eps_0^n$ commutes with $\eps_{n-1}$ and $\eps_1$ is easy. Hence $\zeta\big(  \eJac_b(\xi)\big)=B[\eps_0^n,\eps_1^n,\eps_{n-1}^n]$. From this the first statement is clear. That $\eJac_b(\xi)$ is PI also follows from the finiteness \cite[13.1.13]{McConnellRobson}, and consequently, if prime, $\eJac_b(\xi)$ an order in its quotient ring of fractions \cite[13.6.6]{McConnellRobson} and \cite{Stafford} once again shows the statement on maximality.  
\end{proof}Notice that the centre is independent on $b\in B$. 

The fact that $  \eJac_b(\xi)$ is a PI-ring has some very nice consequences. Recall first that the \emph{Azumaya locus} of $ \eJac_b(\xi)$, $\azu( \eJac_b(\xi))$ is the set of maximal ideals $\mathfrak{m}$ in $\zeta( \eJac_b(\xi))$ such that $ \eJac_b(\xi)/\mathfrak{m}$ is a central simple algebra over $\zeta( \eJac_b(\xi))/\mathfrak{m}$. This is an open subscheme of $\zeta( \eJac_b(\xi))$. The complement is called the \emph{ramification locus}, $\ram( \eJac_b(\xi))$ and is a closed subscheme of codimension one. 

We begin by observing that the commutative points (i.e., the 1-dimensional simple modules) are given as follows. If $\eps_0$, $\eps_1$ and $\eps_{n-1}$ would commute, then we would have the relations
 \begin{align*}
(1 -\xi)\eps_1\eps_0 &= 0,\\
(1-\xi)\eps_0\eps_{n-1} &= 0,\\
 (1-\xi^2)\eps_1\eps_{n-1}&=b\eps_0+b(1-\xi^2).
\end{align*}Now, if $\eps_0=0$ then $\eps_1\eps_{n-1}=b$ and so determine a hyperbola in $\zeta( \eJac_b(\xi))$ via restriction (i.e., $\eps_1^n\eps_{n-1}^n=b$). On the other hand, if $\eps_1=0$ then either $\eps_0=0$, which leads to a contradiction, or $\eps_{n-1}=0$ which, after localization at $b$, gives the point $(-(1-\xi^2)/b,0,0)$. The last case (when $\eps_{n-1}=0$) is exactly the same. This shows that $\ram( \eJac_b(\xi))$ is not empty since the $1$-dimensional simples are in the ramification locus. 
 
One can also show that $\ram( \eJac_b(\xi))$ is the set of maximal ideals of $\zeta( \eJac_b(\xi))$ over which there are more than one maximal ideal in $ \eJac_b(\xi)$. In fact, it is the ramification locus that is the most interesting as this is where the non-commutativity comes in. Therefore, to have a complete description of $\ram(\eJac_b(\xi))$ is very desirable. 

\begin{prop}Put $\mathbf{p}:=(p_1,p_2,p_3)$ with $p_i\in B$. The elements $$\Omega_\mathbf{p}:=p_1\eps_{n-1}\eps_1-p_2\eps_1\eps_{n-1}+p_3\eps_0^2+b\frac{p_2\xi^{-1}-p_1}{1-\xi}\eps_0-b(p_1-p_2),$$ defines a family of normal elements. In fact, $\Omega_\mathbf{p} \cdot x=\tau(x)\cdot\Omega_\mathbf{p}$, with $\tau(\eps_0)=\eps_0$, $\tau(\eps_1)=\xi^2\eps_1$ and $\tau(\eps_{n-1})=\xi^{-2}\eps_{n-1}$. Hence $ \eJac_b(\xi)/(\Omega_\mathbf{p})$ defines a family of non-commutative quadric surfaces in $\ncspec( \eJac_b(\xi))$. 
\end{prop}
\begin{proof}
Tedious, but straightforward, computation. 
\end{proof}
An interesting special case is when $p_2=p_1\xi$. Then we get the pencil of quadrics
$$\Omega_{p_1,p_3}:=p_1\eps_{n-1}\eps_1-\xi p_1\eps_1\eps_{n-1}+p_3\eps_0^2-bp_1(1-\xi).$$The intersection with the center is given by the (commutative) quadric surface
$$\zeta(\Omega_{p_1,p_3})=p_1(1-\xi)\eps_1^n\eps_{n-1}^n+p_3\eps_0^{2n}-bp_1(1-\xi).$$If $\eps_0^{2n}=0$, we see that $\zeta(\Omega_{p_1,p_3})$ includes the hyperbola $\eps_1^n\eps_{n-1}^n=b$ from before, and hence $\Omega_{p_1,p_3}$ intersects the ramification locus. 

From now on we work fibre-by-fibre above $\Lambda$. In other words, we restrict to $B\otimes_{\Lambda} k(\mathfrak{p})\to \eJac_b(\xi)\otimes_\Lambda k(\mathfrak{p})$. Recall the notation $\bar B:=B\otimes_{\Lambda} k(\mathfrak{p})$. We also introduce the notations $$\bar{\eJac}_b(\xi):=\eJac_b(\xi)\otimes_{\Lambda} k(\mathfrak{p}),\quad \bar B^{\mathrm{alg}}:=B\otimes_{\Lambda} k(\mathfrak{p})^{\mathrm{alg}}, \quad \bar{\eJac}_b(\xi)^{\mathrm{alg}}:=\eJac_b(\xi)\otimes_{\Lambda} k(\mathfrak{p})^{\mathrm{alg}}$$ 
\begin{remark}
Notice that if $\xi$ reduces to $1$ modulo $\pp\neq (0)$ then $\bar{\eJac}_b(\xi)^{\mathrm{alg}}$ is the Weyl algebra (over an affine variety) in characteristic $p$, and is therefore Azumaya. If $b=0$ modulo $\pp$ we get a quantum three-space whose geometry (i.e., representation theory) is well understood over $K^\mathrm{alg}$, with $K$ a field of characteristic zero. Over (non-algebraically closed) fields of positive characteristic the geometry seems to be considerably more complicated, and I'm not aware of any results in this direction. 
\end{remark}
\begin{corollary}Let $\phi: \ncspec(\bar{\eJac}_b(\xi))\to \Spec(\zeta(\bar{\eJac}_b(\xi)))$ be the morphism defined by contraction of prime ideals. 
\begin{itemize}
\item[(a)] $\phi$ is a finite, surjective morphism, and every prime of $\bar{\eJac}_b(\xi)$ intersects the centre non-trivially. 
\item[(b)] Then $\mathtt{ram}(\bar{\eJac}_b(\xi)^{\mathrm{alg}})$ is a Zariski-closed subscheme of exact codimension one (i.e., there is at least one component of pure codimension one). 
\end{itemize}
\end{corollary}
\begin{proof}This is standard PI-theory (see for instance \cite{BrownGoodearl} and \cite{McConnellRobson}). The only point which might require a proof is the claim that the ramification locus has exact codimension one. First we notice that $\mathtt{ram}(\bar{\eJac}_b(\xi)^{\mathrm{alg}})$ is not empty by the discussion prior to the statement of the corollary. Recall that a \emph{reflexive Azumaya algebra} is an algebra with ramification locus in codimension two, and reflexive as a module over the centre. Assume that $ \bar{\eJac}_b(\xi)^{\mathrm{alg}}$ is reflexive Azumaya, in particular that $\ram(\bar{\eJac}_b(\xi)^{\mathrm{alg}})$ is strictly in codimension two. Then, since $\bar{\eJac}_b(\xi)^{\mathrm{alg}}$ is Auslander-regular, Cohen--Macaulay and a domain, we have by \cite[Theorem 3.8]{BrownGoodearl_Homological_PI} that the singular locus of $\zeta(\bar{\eJac}_b(\xi)^{\mathrm{alg}})$ coincide with $\ram(\bar{\eJac}_b(\xi)^{\mathrm{alg}})$. However, $\zeta(\bar{\eJac}_b(\xi)^{\mathrm{alg}})$ is a regular scheme and $\ram(\bar{\eJac}_b(\xi)^{\mathrm{alg}})\neq \emptyset$, so we have a contradiction. 
\end{proof}

\begin{remark}We now digress slightly and give an alternative description of $ \eJac_b(\xi)$. A \emph{down-up algebra} $D:=D(\alpha,\beta,\gamma)$, with
$\alpha,\beta,\gamma$ in a commutative ring $S$, is the algebra generated over $S$ by $d$ and $u$
subject to the relations
\begin{align*}
d^2u&=\alpha dud-\beta ud^2-\gamma d,\\
du^2&=\alpha udu-\beta u^2d-\gamma u.
\end{align*}We work with a slightly more general definition here since ordinarily in the definition of down-up algebras one assumes that $S$ is a field. 
\begin{prop}\label{prop:downup}
The algebra defined by {\rm (\ref{eq:envS})} is isomorphic to the
  down-up algebra $D_{\xi}$ over $B$ defined by the relations
\begin{align}\label{eq:qsl2DU}
\begin{split}
d^2u&=\xi(1+\xi)dud-\xi^3 ud^2+a(1-\xi^2)(1-\xi)d,\\
du^2&=\xi(1+\xi)udu-\xi^3 u^2d+a(1-\xi^2)(1-\xi) u.
\end{split}
\end{align}
\end{prop}
\begin{proof}
The proof is very simple. Solve for $a\eps_0$ in (\ref{eq:envS}) and insert in the other two relations. Rename $\eps_{n-1}\to d$ and $\eps_1\to u$ and simplify.
\end{proof}Consequently, when $S$ is a field, some of the above properties could have been deduced from the description of $\eJac_b(\xi)$ as a down-up algebra. 
\end{remark}
\subsubsection{Zeta functions}
In order to study the arithmetic properties of $\eKW_b(\xi)$ and $\eJac_b(\xi)$ we will introduce a \emph{formal} zeta element. The problem with this is that it depends on the ramification locus (in particular, we only work in the PI-case) and so is not easily computed. 

We begin by recalling the set-up in a diagram as follows:
$$
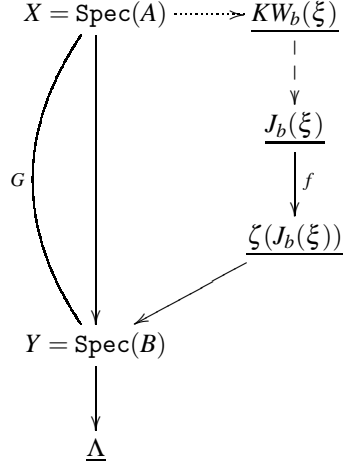
\xymatrix{X=\Spec(A)\ar@{.>}[r]\ar[ddd]\ar@/_2pc/@{-}[ddd]_G & \underline{\eKW_b(\xi)}\ar@{-->}[d]\\
&\underline{\eJac_b(\xi)}\ar[d]^f\\
&\underline{\zeta(\eJac_b(\xi))}\ar[dl]\\
Y=\Spec(B)\ar[d]\\
\underline{\Lambda}}$$

We now define the zeta elements fiber-by-fiber over $\Lambda$ as the formal elements:
\begin{equation}\label{eq:zeta}
\mathbf{z}^f_q(t):=\zeta_{\mathrm{log},\nearrow}^f(t)\boxplus\zeta_{\ram(f)}(t)\boxplus \zeta_{\azu(f)}(t), \quad q=p^k,
\end{equation}where each piece is defined below. This is a purely formal construction, although each piece is in principle computable. 

Let $R$ be a PI-ring with centre $\zeta(R)$ and morphism $f: \underline{R}\to \underline{\zeta(R)}$ (dual to the inclusion), and with ramification locus $\ram(f)$. We now \emph{define}
$$\# \underline{R}(\F_{q^k}):=\sum_{r\in\ram(f)(\F_{q^k})}\Big(\sum_i \dim_{\F_{q^k}}\big(R/\pp_i\big)\Big),$$ where $\pp_i$ are the maximal ideals above $r$. Notice that this is a finite sum. Then we define
$$\zeta_{\ram(f)}(t):=\mathrm{exp}\Big(\sum_{k=1}^\infty \big(\#\underline{R}(\F_{q^k})\big)\frac{t^k}{k}\Big).$$

The piece $\zeta_{\azu(f)}(t)$ is simply the zeta function of the Azumaya locus in the classical sense. 

Now, the last piece, $\zeta_{\nearrow}^f(t)$, is the ``tangent zeta element'' and measures the infinitesimal structure of the fibres. This is in fact the only purely formal part. Let $N$ be the PI-degree of $R$ and form the set $\mathrm{Mat}_{N\times N}(\Z)/\mathbb{S}_{N^2}$, where the symmetric group $\mathbb{S}_{N^2}$ acts on $\mathrm{Mat}_{N\times N}(\Z)$ by permuting the entries in the matrices in the obvious way. We denote the element in $\mathrm{Mat}_{N\times N}(\Z)/\mathbb{S}_{N^2}$, corresponding to $M\in\mathrm{Mat}_{N\times N}(\Z)$, as $[M]$. Notice that $\mathrm{Mat}_{N\times N}(\Z)/\mathbb{S}_{N^2}$ can be identified with the set of unordered $N^2$-tuples of integers. 

We let $\langle \Z(N)\rangle$ denote the free abelian group generated by $\mathrm{Mat}_{N\times N}(\Z)/\mathbb{S}_{N^2}$. For every, $y\in\ram(f)(\F_{q^k})$ we associate an element of $\langle \Z(N)\rangle$ as follows. Let $\pp_1, \pp_2, \dots, \pp_s$, $1\leq s\leq N$, be the $\F_{q^k}$-rational points above $y$, identified with the simple $R$-modules $\mathfrak{m}_i:=R/\pp_i$. Then we form the matrix
$$\Ext_y(k):=\begin{pmatrix}
	e_{11} & e_{12} & \dots & e_{1N}\\
	e_{21} & e_{22} & \dots & e_{2N}\\
	\vdots & \vdots & \ddots & \vdots\\
	e_{N1} & e_{N2} & \dots & e_{NN}
\end{pmatrix}, \quad \text{with}\quad e_{i,j}:=\dim_{\F_{q^k}}\big(\Ext_R^1(\mathfrak{m}_i,\mathfrak{m}_j)\big).$$The entries missing to get a full $N\times N$-matrix are filled with zeros. Now, the element 
$$[\Ext(k)]:=\sum_{y\in\ram(f)(\F_{q^k})}[\Ext_y(k)]\in\langle \Z(N)\rangle$$ is a finite sum and we can form the ``logarithmic tangent-zeta element''
$$\zeta^f_{\mathrm{log},\nearrow}(t):=\sum_{k=1}^\infty \Big(\sum_{y\in\ram(f)(\F_{q^k})}[\Ext_y(k)]\Big)\frac{t^k}{k}=\sum_{k=1}^\infty [\Ext(k)]\frac{t^k}{k}.$$

There is a natural ``trace function'', $T:\langle \Z(N)\rangle \to \Z$, induced from the 
function
$$\mathrm{Mat}_{N\times N}(\Z)\to \Z,\quad (m_{ij})\mapsto \sum_{i,j} m_{ij}$$ and extended linearly. Applied to $[\Ext(k)]$ we get
$$T([\Ext(k)])=\sum_{i,j}^N\dim_{\F_{q^k}}\big(\Ext_R^1(\mathfrak{m}_i,\mathfrak{m}_j)\big)=\sum_{i,j}^N e_{ij}.$$We can now form the ``tangent zeta function'' as
$$\zeta_{\nearrow}^f(t):=\mathrm{exp}\Big(\sum_{k=1}^\infty T([\Ext(k)])\frac{t^k}{k}\Big).$$Clearly, $\zeta_{\mathrm{log},\nearrow}^f(t)$ encodes more information than $\zeta_{\nearrow}^f(t)$, as all the tangent dimensions are explicitly given (not just the sum of the dimensions).
\begin{remark}In order to compute the zeta element $\mathbf{z}^f_q$ we obviously need to know the ramification locus of $f$ as a first step. However, in general this is a significant one, especially since we need this over a non-algebraically closed field. This shows that the zeta element is not easily determined. On the other hand, it should be clear that there is a lot of arithmetic information encoded in this invariant and I suspect that it might be worthwhile to overcome the ramification hurdles. 
\end{remark}Finally we define the \emph{arithmetic zeta element} associated with $R$ to be the (formal) product
$$\mathbf{z}^f_{\mathrm{ar}}:=\prod_{\pp\in\underline{\Lambda}}\mathbf{z}^f_q, \quad q:=k(\pp).$$

\vspace{0.1cm}
\begin{center}\rule{0.5\textwidth}{0.5pt}\end{center}
\vspace{0.3cm}

\subsection*{Final remark: Rational points on abelian varieties} 
Let $A/k$ be an abelian variety over a field $k$. Consider the group algebra $k[A]$, i.e., the algebra generated over $k$ by the closed points as basis. This is clearly a commutative, associative, ring with unity. Notice that if $K\supset k$ and $K'\supset k$ are different extensions of $k$, the $K$-rational points and $K'$-rational points are different so $k[A(K)]$ and $k[A(K')]$ are also different. Therefore we need to specify which closed points we mean. 

Now, $\partial_{\sigma^n}:=\id-\sigma^n$, where $\sigma\in\Gal(k^{\text{sep}}/k)$ and $n\geq 1$, acts on $k[A(K)]$, for any $k^{\text{sep}}\supseteq K\supseteq k$ (trivially on $k[A(k)]$). This operator can be viewed as measuring the effect $\sigma^n$ has on the $K$-rational points on $A$, and is as we now know a $\sigma^n$-twisted derivation. Therefore we can form the hom-Lie algebra $k[A(K)]\cdot\partial_{\sigma^n}$ and the associated equivariant hom-Lie algebra. Thus this structure not only, encapsulates the Galois-theoretic properties of the rational points of $A$, but also the dynamics of the family $\{\sigma^n\}_{n\in \mathbb{N}}$. 

Of particular interest here is also different subgroups, viz., $m$-torsion points $A[m](K)$, the associated Barsotti--Tate groups, and their induced equivariant hom-Lie algebras. 
 
The further investigations of the above topics is beyond the intended scope of this paper and have to wait for another time.

\vspace{0.1cm}
\begin{center}\rule{0.5\textwidth}{0.5pt}\end{center}
\vspace{0.3cm}

\bibliographystyle{alpha}
\bibliography{refarithom}
\end{document}